\documentclass[11pt]{article}
\usepackage[utf8]{inputenc}
\usepackage[english]{babel}
\usepackage{amsfonts}
\usepackage{amsmath,amsthm,amscd,amssymb,mathrsfs,setspace}
\usepackage{mathtools}
\usepackage{cite}
\usepackage{latexsym,epsf,epsfig}
\usepackage{color}
\usepackage[a4paper,top=1in,bottom=1in,left=1in,right=1in,marginparwidth=0.5in]{geometry}
\usepackage[colorlinks,citecolor=lccx]{hyperref}
\usepackage{dsfont}
\usepackage{comment}
\usepackage[nameinlink,capitalise]{cleveref}
\usepackage[mathscr]{euscript}
\usepackage{dutchcal}
\usepackage{algorithm}
\usepackage{algorithmicx}
\usepackage{algpseudocode}
\usepackage[table]{xcolor}
\usepackage{soul}
\usepackage{subcaption}
\usepackage{siunitx}
\usepackage{tikz}
\usetikzlibrary{decorations.pathreplacing,calligraphy}
\usepackage{hyperref}
\usepackage{comment}
\usepackage{booktabs}
\usepackage{multirow}
\usepackage{adjustbox}
\usepackage{authblk}
\usepackage{comment}
\usepackage{enumitem}
\usepackage{xfrac}
\usepackage{marginnote}

\definecolor{lccx}{HTML}{92268F}

\ifpdf
\hypersetup{
	pdftitle={Coefficient identification manuscript},
	pdfauthor={},
	linkcolor=lccx
}
\fi

\newcommand{\N}{\mathbb{N}}
\newcommand{\R}{\mathbb{R}}

\newcommand{\calC}{\mathcal{C}}
\newcommand{\calM}{\mathcal{M}}

\newcommand{\calF}{\mathcal{F}}
\newcommand{\calL}{\mathcal{L}}
\newcommand{\calH}{\mathcal{H}}
\newcommand{\maxrho}{\sideset{}{_\rho}{\max}}

\DeclareMathOperator*{\BV}{BV}

\DeclareMathOperator*{\TV}{TV}

\newcommand*\dd{\mathop{}\!\mathrm{d}}

\DeclareMathOperator*{\dvg}{div}
\DeclareMathOperator{\supp}{supp}
\DeclareMathOperator*{\pred}{pred}
\DeclareMathOperator*{\ared}{ared}

\newcommand{\weakto}{\rightharpoonup}
\newcommand{\weakstarto}{\stackrel{\ast}{\rightharpoonup}}
\newcommand{\feas}{U_{\textrm{ad}}}

\newtheorem{theorem}{Theorem}[section]
\newtheorem{proposition}[theorem]{Proposition}
\newtheorem{lemma}[theorem]{Lemma}
\newtheorem{corollary}[theorem]{Corollary}

\theoremstyle{definition}
\newtheorem{definition}[theorem]{Definition}
\newtheorem{assumption}[theorem]{Assumption}

\theoremstyle{remark}
\newtheorem{remark}[theorem]{Remark}

\crefname{assumption}{Assumption}{Assumptions}
\Crefname{assumption}{Assumption}{Assumptions}

\title{Optimizing the Principal Coefficient of
Elliptic Equations using $L^p$-regularity,
$p < \infty$}

\author[1]{Ala' Alalabi}
\author[1]{Lorena Bociu}
\author[2]{Paul Manns}

\affil[1]{Department of Mathematics, NC State University, North Carolina, USA\\

\textit{awalalab@ncsu.edu, lvbociu@ncsu.edu}}

\affil[2]{Department of Mathematics, TU Dortmund University, Dortmund, Germany\\
\textit{paul.manns@tu-dortmund.de}}
\begin{document}
\maketitle
\begingroup
\renewcommand\thefootnote{\fnsymbol{footnote}}
\footnotetext[0]{Lorena Bociu acknowledges funding by NSF-DMS 2510261.
Paul Manns acknowledges funding by DFG under project no.~540198933.}
\endgroup

\begin{abstract}
We study coefficient identification problems for elliptic partial differential equations with total variation regularization and control constraints. Existing related literature relies on continuity and differentiability properties of the control-to-state operator with respect to the $L^\infty$-norm. While this is sufficient for deriving optimality conditions, it is not well-suited for numerical algorithms, as it neglects the spatial extent of perturbations and leads to a qualitative discrepancy compared to $L^q$-norms with $q < \infty$. In this work, we address this gap by exploiting $W^{1,s}$-regularity results to establish differentiability properties of the control-to-state operator with respect to $L^q$-norms for finite $q$. Based on this framework, we derive first- and second-order differentiability results for the reduced objective functional and establish first-order optimality conditions involving a restricted subdifferential characterization of the total variation seminorm and corresponding regularity of the associated multipliers. Building on this, we analyze a nonsmooth trust-region method based on an $L^r$-trust region for $r > 0.5 q$ and prove its convergence to first-order stationary points.
\end{abstract}

\textbf{Keywords:} Elliptic partial differential equations,
coefficient identification,
PDE-constrained optimization,
total variation regularization,
trust-region methods.

\section{Introduction}\label{sec:introduction}
Let $\Omega \subset \R^d$, $d \ge 2$, be a bounded domain with Lipschitz boundary $\partial \Omega$. We consider optimization problems of the form
\begin{gather}\label{eq:p}
 \min_{y,u}\ [F(y) + \TV(u)]
\ \text{subject to}\ 
\left\{
\begin{aligned}
- \dvg(  u \nabla y) &= f \text{ in } \Omega, \\
y &= f_D \text{ on } \tilde{\Gamma}, \\
 u \nabla y \cdot \nu  &= 0  \text{ on } \Gamma, \\
0 < u_\ell \le u(x) &\le u_u \text{ for a.e.\ }x \in \Omega,
\end{aligned}
\right.
\tag{P}
\end{gather}
with given source terms $f \in W^{-1,2}(\Omega)$ and  $f_D \in W^{1/2,2} (\tilde{\Gamma})$.
The boundary $\partial \Omega$ is $C^{1,\alpha}$, the vector field $\nu$ denotes the unit outer normal on ${\Gamma}$, and the non-empty sets
$\tilde{\Gamma}$, $\Gamma$ are
closed and non-negligible in $\partial \Omega$.
The boundary setting corresponds to the one provided in Theorem 7.36.5
in \cite{medkova2018laplace} and in particular
satisfies the conditions
of Definition 2 in \cite{groger1989aw}, which
\emph{means, roughly speaking, that $\Gamma$ and 
$\tilde{\Gamma}$ are separated by a Lipschitzian 
hypersurface of $\Omega$} \cite{groger1989aw} 
(see also the discussion in Appendix A of 
\cite{clason2018total} and
\cite{haller2009holder,rehberg2015holder,haller2024holder} 
for further characterizations and refined results). 

We assume that the cost functional $F : W^{1,2}(\Omega) \to \R$ is weakly lower semi-continuous and bounded below
and, as usual, $\TV : L^1(\Omega) \to [0,\infty]$ denotes the total variation seminorm.
The positive scalars $u_\ell$ and $u_u$ define uniform box constraints on the control  $u$
and we denote the set of feasible controls with finite objective value by
\[ \feas = \{ u \in \BV(\Omega)\,:\, u(x) \in [u_\ell,u_u] \text{ for a.e.\ }x \in \Omega \}. \]

A closely related optimization problem is studied in \cite{clason2018total}, while \cite{assmann2013identification} considers the same setting with $H^s$-regularization 
($s < 1$) instead of $\TV$-regularization. 
The recent article \cite{antil2024uniform} 
also studies a related problem in topology optimization, where regularization is achieved not through a penalty term in the objective, but via a so-called \emph{Helmholtz filter}\cite{lazarov2011filters} that smooths the control $u$
before it is inserted into the PDE. 
 All of these regularizations recover the lack of compactness
of the solution operator to the PDE that is due to the product $u \nabla y$. What these
articles as well as others addressing related problems (see, e.g., \cite{garcke2022phase})
have in common is that substantial portions, and in some cases the entirety, of their analysis rely on continuity and differentiability of the control-to-state operator that maps $u$ to the
corresponding $y(u)$ with respect to the $L^\infty(\Omega)$-norm on the feasible set $\feas$.
Although it is well known in the community (see, e.g., \cite{clason2018total,clason2021optimal})
that, following the seminal work of Gröger \cite{groger1989aw}, 
improved $W^{1,s}$-regularity theory for $s > 2$ is available, this theory has so far seen little use in the optimal control and topology optimization literature for establishing the relevant continuity and differentiability properties.
We believe that this may be explained by the fact that the optimality conditions satisfied at the points one aims to recover are essentially unchanged, whether or not additional regularity of the control-to-state operator has been verified.

The situation changes, however, when one aims to prove (global) convergence of nonlinear programming-type algorithms to stationary points. If a globalization mechanism such as a trust region or a proximal operator is employed, the analysis must exploit regularity with respect to the norm associated with the trust region or proximal operator.
We believe that using the $L^\infty$-norm for this purpose is impractical, because it accounts only for the maximal pointwise height difference between two successive control iterates, while completely neglecting the measure of the region on which this difference occurs. In particular, a step that changes the control from $u_\ell$ to $u_u$
 on an arbitrarily small subset of $\Omega$ of positive measure is treated as being of the same size as a step that changes the control from $u_\ell$ to $u_u$
 on the whole domain $\Omega$. We highlight that the $L^\infty$-norm
is the only one with this behavior so that there is a qualitative gap between using the
an $L^r$-norm with $r < \infty$ and $r = \infty$. This qualitative gap can also be
observed in the behavior of smoothing operators via mollification, where one
can recover a fixed limiting control in $L^p$ for all $p \in [1,\infty)$
but not for $p = \infty$.

In this article we address this gap in the literature and make the following contributions.
Using the aforementioned regularity theory of Gröger \cite{groger1989aw}, 
we provide the first-order derivative of the reduced objective functional
$F \circ S$, where the control-to-state operator $S$ is a map $L^q(\Omega) \to \R$, for some $q < \infty$. Achieving this improved regularity requires
 slightly higher regularity of the
source terms (see \cref{prop:Fu_W-1-s} and \cref{prp-new:elementary_properties})
and a concession on a mild norm gap in the derivative. For second-order differentiability, additional assumptions are needed, and the required integrability exponent of the input must be increased further, though it can still be chosen finite.
We establish a first-order optimality condition for \eqref{eq:p} in $\BV(\Omega) \cap L^q(\Omega)$ that involves a 
restricted subdifferential inequality for the nonsmooth but convex total variation 
seminorm as well as $L^{q'}$-regularity of the multipliers for the box constraints in 
$\feas$, where $q'$ is the H\"older conjugate of $q$. To this end, we advance the 
proof strategy from \cite{natemeyer2022penalty} by cutoff arguments to handle the 
lower regularity of the objective and multipliers. 

Furthermore, we prove the  convergence of a  non-smooth trust-region algorithm to 
first-order optimal points by using a trust-region ball with respect to the 
$L^r$-norm with $r > 0.5 q$. If second-order information is
used in the model function of the trust-region algorithm, one may also need to 
increase $q$ due to the higher integrability index required for the second derivative.
To this end, we follow the high-level proof strategy of \S11 in \cite{conn2000trust}, but adapt it to accommodate the fact that, in our setting, $\TV$-seminorm  is not Lipschitz continuous with respect to the  $L^q$-norm.
Specifically, we show how the Cauchy decrease condition 
therein can be substituted by descent properties that arise from the maximal 
monotonicity of the subdifferential of the $\TV$-seminorm.

In the interest of a balanced discussion, we note that $q$
is generally much larger than $2$, so that implementing the trust-region subproblems with a value
covered by the theory may be challenging in practice. Even so, our analysis bridges the
qualitative gap between the $L^\infty$ case and the  $L^q$ setting for finite 
$q$. Bulding on these results, one may, for example, introduce
mollification if additional regularity is required.

The remainder of the article is structured as follows.
Section~\ref{Section2} introduces the notation used throughout the paper,
presents the state and adjoint equations and establishes regularity properties of the control-to-state operator. Moreover, it  establishes first- and second-order differentiability of the control-to-state operator and the reduced objective functional.
Section~\ref{sec:first-order_opt} derives first-order optimality conditions for the coefficient identification problem.
Finally, Section~\ref{sec:tr_algorithm} presents a nonsmooth trust-region framework and proves its convergence to first-order stationary points.

\section{Analysis of State and Adjoint Equations} \label{Section2}
We begin the section by introducing the notation used throughout 
the paper.

Let $C$ be a convex set and $X$ be a Banach space with norm $\|\cdot\|_X$.
We say that a function $j : C \to \R$ is differentiable at $u \in C$ with respect to $\|\cdot\|_X$
if there exists $j'(u) \in X^*$ such that
\[
\underset{\substack{\phi \to 0 \\ u + \phi \in C}}
\lim \frac{\big|j(u + \phi) - j(u) - \langle j'(u), \phi\rangle_{X^*,X}\big|}{\|\phi\|_{X}} = 0.
\]
\begin{remark}
Although we define differentiability in terms of feasible variations around the linearization point $u$, we note that one could alternatively embed the feasible set into a larger set that is open in $L^\infty(\Omega)$
and then work with this enlarged set as it is done, for example, in \cite{clason2018total}. However, this does not eliminate the main difficulties in our analysis, which arise from the fact that the differentiation at $u$ is performed with respect to $L^q(\Omega)$-norm, with $q < \infty$.
\end{remark}

If $X = L^p(\Omega)$ for $p \in [1,\infty)$, we write
$\nabla_R j(u)$ for the representative of $j'(u) \in X^*$ in $L^{p'}(\Omega)$, where $p'$ is the Hölder conjugate of $p$.

We recall that for an open set $\Omega$, $\calM(\Omega,\R^d)$ is
the space of regular Borel measures with finite variation
norm, which is the dual space of $C_0(\Omega,\R^d)$.

For an integrability index $s \in [1,\infty)$, we denote its Hölder conjugate index by
$s' = s/(s-1) \in (1,\infty]$.
By $W^{1,s}_{\tilde\Gamma}(\Omega)$ we denote the closed subspace of 
$W^{1,s}(\Omega)$ consisting of functions whose trace vanishes on $\tilde\Gamma$, and we let 
 $W^{-1,s}(\Omega)$ be the dual of $W^{1,s'}_{\tilde\Gamma}(\Omega)$.
We use $\|\cdot\|_{W^{1,s}_{\tilde\Gamma}(\Omega)} = \|\cdot\|_{W^{1,s}(\Omega)}$ as the norm on $W^{1,s}_{\tilde\Gamma}(\Omega)$.
In the interest of a clean exposition, we often drop the domain of a space
when stating the Lebesgue or Sobolev norm in $\Omega$.

\subsection{Analysis of State Equation}\label{S2.1}
Recall our state equation
\begin{equation}\label{eq:pde}
\begin{cases}
    - \dvg(u \nabla y) = f & \text{ in } \Omega\\
y = f_D & \text{ on } \tilde{\Gamma}\\
u \nabla y \cdot \nu = 0 &\text{ on } \Gamma
\end{cases}
\end{equation}
with $\Omega$, $\Gamma$, $\tilde{\Gamma}$ satisfying the assumptions listed in  \cref{sec:introduction}. For some $s > 2$ that will be asserted below in \cref{prp-new:elementary_properties}
and \cref{rem:index_s}, we will mainly consider solutions
in the state space $W^{1,s}_{\tilde\Gamma}(\Omega)$ and the data $f$ in the
space $W^{-1,s}(\Omega)$. In order to homogenize the Dirichlet boundary condition in \eqref{eq:pde}, we assume that 
$f_D \in  W^{1-1/s,s}(\tilde{\Gamma})$ and consider the
Laplace equation with mixed boundary conditions 
\begin{equation}\label{eq:Dirichlet_map}
    \begin{cases}
     -\Delta w = 0 &\text{ in } \Omega \\
w = f_D &\text{ on } \tilde{\Gamma} \\
\nabla w \cdot \nu = 0 &\text{ on }\Gamma   
    \end{cases}
\end{equation}

We recall the following wellposedness result (found in Theorem 7.36.5 in \cite{medkova2018laplace}), which provides a harmonic extension of the Dirichlet data in \eqref{eq:Dirichlet_map}. 
\begin{proposition}\label{prp:Dirichlet_map}
Let $p \in (1,\infty)$ and $f_D \in W^{1-1/p,p}(\tilde{\Gamma})$.
Under the domain assumptions stated in
\cref{sec:introduction},
system
\eqref{eq:Dirichlet_map} has a unique solution $\tilde{f}_D \in W^{1,p}(\Omega)$ satisfying
\begin{equation}\label{est-eq_from19}
    \|\tilde{f}_D\|_{W^{1,p}}
   \le C(p) \|f_D\|_{W^{1-1/p,p}(\tilde\Gamma)},
\end{equation}
for some constant $C(p) > 0$ that depends on $\Omega$, $\Gamma$, $\tilde\Gamma$,
and $p$.
\end{proposition}

Using \cref{prp:Dirichlet_map} with $p = s$, we know that there exists a unique solution $\tilde f_D \in   W^{1,s}(\Omega)$ of \eqref{eq:Dirichlet_map}. 
Then $y$ is a solution for \eqref{eq:pde} if and only if 
$\displaystyle \tilde y \coloneqq y - \tilde f_D$ solves the boundary value problem for given $f \in W^{-1,s}(\Omega)$:
\begin{equation}\label{eq:pde_homogenized}
    \begin{cases}
- \dvg(u \nabla \tilde y)
= f + \dvg(u\nabla \tilde f_D)
& \text{ in }\Omega\\
\tilde y = 0
&\text{ on }\tilde\Gamma\\
u \nabla \tilde y\cdot\nu
= 0
&\text{ on }\Gamma
    \end{cases}
\end{equation}

We can equivalently rewrite system  \eqref{eq:pde_homogenized}  
as
\[ E(\tilde y,u)=0, \qquad \text{with} \ 
E:W^{1,s}_{\tilde\Gamma}(\Omega)\times \feas\to W^{-1,s}(\Omega), \qquad
E(\tilde y,u)\coloneqq A(u)\tilde y - F_u,
\]
with  operator $A(u) : W^{1,s}_{\tilde\Gamma}(\Omega) \to W^{-1,s}(\Omega)$
 defined by $
\langle A(u)\tilde y , v\rangle
\coloneqq \int_\Omega u\nabla \tilde y\cdot\nabla v\,dx
\; \text{ for all }v \in W^{1,s'}_{\tilde\Gamma}(\Omega)
$
and $F_u$ is given by
\begin{gather}\label{eq:Fu_definition}
\langle F_u , v\rangle
\coloneqq
\langle f, v\rangle
- \int_\Omega u\nabla\tilde f_D\cdot\nabla v\,dx
\quad\text{ for all }v \in W^{1,s'}_{\tilde\Gamma}(\Omega).
\end{gather}

We have the following  estimate for the norm of $F_u$ introduced above.

\begin{proposition} \label{prop:Fu_W-1-s}
    Let $s=2$, or let $s \in (2,\infty)$ be the elliptic regularity exponent from \cite{groger1989aw} and \cite{clason2018total}. Assume that $f\in W^{-1,s}(\Omega),
\; 
f_D\in W^{1-1/s,s}(\widetilde{\Gamma}).$ Then, for every $u\in U_{\mathrm{ad}}$, the functional $F_u$ belongs to
$W^{-1,s}(\Omega)$ and satisfies 
\begin{equation}\label{eq:Fu_W-1-s_estimate}
\|F_u\|_{W^{-1,s}}
\le
C(s)\bigl(\|f\|_{W^{-1,s}} + \|f_D\|_{W^{1-1/s,s}(\tilde\Gamma)}\bigr)
\end{equation}
where $C(s)>0$  depends on
$u_u$, $\Omega$, $\Gamma$,
$\widetilde{\Gamma}$, and $s$.
\end{proposition}
\begin{proof}
Let $v\in W^{1,s'}_{\tilde\Gamma}(\Omega)$. Using the fact that $u\in L^\infty(\Omega)$ and estimate \eqref{est-eq_from19} with $p=s$, we obtain
\[
|\langle F_u,v\rangle|
\leq \|f\|_{W^{-1,s}} \|v\|_{W^{1,s'}_{\tilde\Gamma}}
+
C(s)\|u\|_{L^\infty}
\|f_D\|_{W^{1-1/s,s}(\tilde\Gamma)}
\|v\|_{W^{1,s'}_{\tilde\Gamma}},
\]
which provides the desired estimate  \eqref{eq:Fu_W-1-s_estimate} on the norm of $F_u$. 
\end{proof}

\subsection{Control-to-state Operator}
Using the setup introduced in Section \ref{S2.1}, the control-to-state operator of BVP \eqref{eq:pde}
is given by
\begin{equation}\label{dfnS}
S:\feas\to W^{1,2}(\Omega),
\qquad 
S(u) \coloneqq \tilde f_D + A(u)^{-1}F_u,
\end{equation}
where we used the fact that the operator  $A(u):W^{1,2}_{\tilde\Gamma}(\Omega)\to W^{-1,2}(\Omega) $
is invertible by standard elliptic regularity in the case where  a homogeneous Dirichlet
condition is imposed on a $\calH^{d-1}$-non-negligible subset
of $\partial \Omega$ (see, e.g.,
\S1.2.6 and \S1.3.2 in \cite{necas2011direct}). In fact, we obtain improved regularity  due to the fact that $A(u)$ 
actually maps  $W^{1,s}_{\tilde\Gamma}(\Omega)$ onto 
$W^{-1,s}(\Omega)$, for some $s > 2$. To this end, we recall the following results on the operators $A(u)$ and $S$ from \cite{clason2018total}; see, in particular
Proposition 2.3 therein; which themselves are based on the elliptic regularity theory provided in \cite{groger1989aw}. Since we need them
here for our mixed boundary condition, we state them again and provide a brief proof below.
\begin{proposition}\label{prp-new:elementary_properties}
Let $f\in W^{-1,2}(\Omega)$, $f_D\in W^{1/2,2}(\tilde\Gamma)$. Then:
\begin{enumerate}
\item There exists $C>0$ that depends on $u_u$, $\Omega$,
$\Gamma$, and $\tilde{\Gamma}$ such that for all $u \in \feas$:
\[
\|S(u)\|_{W^{1,2}} \le C\bigl(\|f_D\|_{W^{1/2,2}(\tilde\Gamma)} + 
\|f\|_{W^{-1,2}}\bigr).
\]
\item There exists $L>0$ such that for all $u_1$, $u_2 \in \feas,$ $\|S(u_1)-S(u_2)\|_{W^{1,2}} \le L\|u_1-u_2\|_{L^\infty}. $\end{enumerate}
Moreover, there exist $s \in (2,\infty)$ and $[2,s] \ni \sigma \mapsto C(\sigma) \in [0,\infty)$ (that depends on $u_u$, $\Omega$, $\Gamma$, $\tilde{\Gamma}$, and $s$)
such that for all $\sigma \in [2,s]$ with $f_D\in W^{1-1/\sigma,\sigma}(\tilde\Gamma)$, $f \in W^{-1,\sigma}(\Omega)$, and all $u \in \feas$:
\begin{gather}\label{eq:w1s_estimate}
    \|S(u)\|_{W^{1,\sigma}}
\le C(\sigma)\bigl(\|f_D\|_{W^{1-1/\sigma,\sigma}(\tilde{\Gamma})} + \|f\|_{W^{-1,\sigma}}\bigr)
\end{gather}
\end{proposition}
\begin{proof}
The nonstandard assertion is the $W^{1,\sigma}$-regularity estimate. Let $
G:=\Omega\cup\Gamma$. 
By the assumptions listed in  \cref{sec:introduction}, the set $G$ is regular in the sense of Gröger \cite[Definition 2]{groger1989aw}. Hence Gröger's regularity theory for mixed boundary value problems applies. More precisely, following the proof of \cite[Theorem~1 and Lemma~1]{groger1989aw}, there exists an exponent $s>2$ such that, for every $\sigma\in[2,s]$ and every $u\in U_{\rm ad}$, the operator
$A(u): W^{1,\sigma}_{\widetilde\Gamma}(\Omega)\to W^{-1,\sigma}(\Omega)
$
is an isomorphism. In \cite{groger1989aw}, the admissible exponent $\sigma$ is characterized through a  condition of the form
$
M_\sigma k<1,$ 
where $M_\sigma$ denotes the norm of the inverse of a suitable  operator and
$
k=\left(1-{u_\ell^2}/{u_u^2}\right)^{1/2}.
$ Later, Herzog et al.~\cite[Theorem~1.1]{herzog2011integrability} showed that, since $M_2=1$ and $k<1$, the condition $M_\sigma k<1$ remains valid for all exponents $\sigma$ sufficiently close to $2$. In particular, there exists an interval $[2,s]$ on which the isomorphism property is valid for all $\sigma \in [2,s]$. For a detailed proof and a more general setting, we also
refer to \cite[Theorem 1.1]{groger1989aw}. 
\end{proof}

\begin{remark}
By inspecting the statement and proof of Theorem 1 in \cite{groger1989aw}, 
which is the center point for the arguments that lead to 
\cref{prp-new:elementary_properties}, we can 
see that the value of $s$ may depend on the specific value of the coercivity
constant, namely $u_\ell$. Consequently, every continuity and
differentiability analysis of $S$ that is based on this regularity must 
ensure that variations of feasible points remain feasible with respect to 
this lower bound on the control because otherwise the norm with respect to which
one can analyze the problem would change.
\end{remark}
\begin{remark}\label{rem:index_s} For the remainder of the paper, whenever we refer to the integrability index $s$, we mean the one asserted in \cref{prp-new:elementary_properties}.
\end{remark}

Next we show that \cref{prp-new:elementary_properties} implies the Lipschitz 
continuity for the solution operator 
$S : (\feas, \|\cdot\|_{L^{q}}) \to W^{1,2}(\Omega)$ for large enough values of $q$.
\begin{lemma}\label{lem:S_Lipschitz}
Let $f_D\in W^{1-1/s,s}(\tilde\Gamma)$,
$f \in W^{-1,s}(\Omega)$.
For all $q \in \big[\frac{2s}{s-2},\infty\big)$, the operator
$
S : (\feas,\|\cdot\|_{L^{q}}) \to W^{1,2}(\Omega)
$
is Lipschitz continuous with Lipschitz constant depending on $q$.
\end{lemma}
\begin{proof}
Let $u_i \in\feas$ and set $y_i \coloneqq S(u_i)$,
$\tilde y_i \coloneqq y_i-\tilde f_D \in W^{1,s}_{\tilde\Gamma}(\Omega)$ for $i \in \{1,2\}$. Then $\tilde y_i$ solves
\[
\int_\Omega u_i \nabla \tilde y_i\cdot \nabla v\,dx = \langle F_{u_i}, v\rangle
\quad\text{ for all } v\in W^{1,s'}_{\tilde\Gamma}(\Omega)
\]
with $F_{u_i}$ defined in \eqref{eq:Fu_definition}. Subtracting the two equations yields the identity
\[
\int_\Omega u_1 \nabla(\tilde y_1-\tilde y_2)\cdot\nabla v\,dx
=
\langle F_{u_1}-F_{u_2}, v\rangle
+
\int_\Omega (u_2-u_1)\nabla \tilde y_2\cdot\nabla v\,dx ,
\]
which is equivalent to 
$
\int_\Omega u_1 \nabla(y_1-y_2)\cdot\nabla v\,dx
=
\int_\Omega (u_2-u_1)\nabla y_2\cdot\nabla v\,dx.
$
If we test with $v = \tilde{y}_1 - \tilde{y}_2$ we obtain
\[
\bigl(u_1 \nabla (y_1- y_2), \nabla(y_1-y_2)\bigr)_{L^2(\Omega)}
= 
\bigl((u_2-u_1)\nabla y_2,\nabla(y_1-y_2)\bigr)_{L^2(\Omega)}.
\]
Estimating the left-hand side from below and the right-hand side
from above implies
\begin{align*}
u_\ell\|\nabla(y_1 - y_2)\|_{L^2}^2
&\le \|\sqrt{u_1}\nabla(y_1 - y_2)\|_{L^2}^2 \le \|\nabla(y_1 - y_2)\|_{L^2} \|\nabla y_2\|_{L^s} \|u_1 - u_2\|_{L^{2s/(s-2)}} \\
&\le C_1(s,f_D,f)\|\nabla(y_1 - y_2)\|_{L^2} \|u_1 - u_2\|_{L^{2s/(s-2)}},
\end{align*}
where
$
C_1(s,f_D,f) \coloneqq
C(s)\bigl(\|f\|_{W^{-1,s}(\Omega)} + \|f_D\|_{W^{1-1/s,s}(\tilde\Gamma)}\bigr)$. 
From the definition of $S$ in \eqref{dfnS}, we have that $y_1 - y_2 =
\bigl(A(u_1)^{-1}F_{u_1}-A(u_2)^{-1}F_{u_2}\bigr)
\in W^{1,2}_{\tilde\Gamma}(\Omega),$ 
and hence, by the Poincar\'e inequality on $W^{1,2}_{\tilde\Gamma}(\Omega)$
with constant $C_2(\Omega,\tilde\Gamma)$ (see, e.g., Theorem 1.9 in \cite{necas2011direct}) we obtain that
\[
\|y_1 - y_2\|_{L^2}
\le
C_2(\Omega,\tilde\Gamma)\|\nabla(y_1-y_2)\|_{L^2},
\]
which yields the desired Lipschitz continuity.
\end{proof}
\begin{remark}
Since in practice the admissible $s \in (2,\infty)$ is usually close to $2$,
the condition $q \in \big[\frac{2s}{s-2},\infty\big)$ may force $q$ to be rather large.
\end{remark}

\subsection{Adjoint Equation}

General adjoint calculus (formally) gives (under sufficient differentiability assumptions):
\[ (F \circ S)'(u) = -
F'(S(u))
\bigl(\partial_1 E(S(u)-\tilde f_D,u)\bigr)^{-1}\,
\partial_2 E(S(u)-\tilde f_D,u).
\]
Since
$(\partial_1 E(S(u)-\tilde f_D,u))^* = A(u)$
must hold when both sides are well defined, we can introduce the adjoint 
state $p = p(y,u) \in W^{1,2}_{\tilde{\Gamma}}(\Omega)$ for $y = S(u)$
as the solution to
$ A(u)p = -F'(y) \in W^{-1,2}(\Omega), $
which in variational form is given by
\[
\int_\Omega u\,\nabla p\cdot\nabla v\,\dd x
= \langle - F'(y), v\rangle
\quad\text{ for all }v\in W^{1,2}_{\tilde{\Gamma}}(\Omega),
\]
and corresponds to the mixed boundary value problem
\begin{equation}\label{eq:adjoint}
\begin{cases}
\operatorname{div}(u\nabla p) = - F'(y)
& \text{in }\Omega\\
p = 0 &\text{on }\tilde\Gamma\\
u \nabla p\cdot\nu = 0 &\text{on }\Gamma.
\end{cases}
\end{equation}
We note that the regularity considerations for the state
equation also apply here and yield a solution
$p\in W^{1,s}_{\tilde\Gamma}(\Omega)$ for the same exponent $s>2$ as in \cref{prp-new:elementary_properties}, 
if we assume that $-F'(y) \in W^{-1,s}(\Omega)$ holds. Thus, the implied adjoint solution
operator
\[
S_p : W^{1,s} (\Omega) \times \feas  \to  W^{1,s}_{\tilde\Gamma}(\Omega),\ \  S_p(y,u) = p
\]
is well defined and bounded provided that $-F'(y)$ is uniformly
bounded in $W^{-1,s}(\Omega)$ for $y \in \{ S(u) : u \in \feas\}$.

\begin{lemma}\label{lem:SpLipschitz}
Let $F: W^{1,2}(\Omega)\to \R$ be continuously differentiable
with  Lipschitz continuous derivative $F' : W^{1,2}(\Omega) \to W^{-1,2}(\Omega)$
and $\sup\{ \|F'(S(u))\|_{W^{-1,s}} : u \in \feas\} < \infty$.
Then the following properties hold:
\begin{enumerate}
\item For all $q \in \big[\frac{2s}{s-2},\infty\big)$ and all $y$ with $-F'(y) \in W^{-1,s}(\Omega)$, the map $ 
S_p(y,\cdot): (\feas, \|\cdot\|_{L^q}) \to W^{1,2}_{\tilde\Gamma}(\Omega)$ 
is Lipschitz continuous with Lipschitz constant depending on $q$,
$\|F'(y)\|_{W^{-1,s}}$, and the coercivity constant $u_\ell$.

\item For all $u \in \feas$, the map $
S_p(\cdot,u) : W^{1,2}(\Omega) \to W^{1,2}_{\tilde\Gamma}(\Omega)$
is Lipschitz continuous.
\end{enumerate}
\end{lemma}
\begin{proof}
The first claim follows from arguments similar to those in \cref{lem:S_Lipschitz}.
The second claim follows from the assumed Lipschitz continuity of $F'$
and \cref{prp-new:elementary_properties}.
\end{proof}
\begin{remark}
    We note that the Lipschitz continuity of $F'$ implies uniform boundedness of  $F'(S(u))$ in $W^{-1,2}(\Omega),$ but not necessarily in $W^{-1,s}(\Omega)$ for $s \in (2,\infty)$. The latter is therefore imposed to ensure uniform $W^{1,s}$-regularity of the adjoint state.
\end{remark}

The next corollary follows directly from \cref{lem:SpLipschitz}.
\begin{corollary}\label{cor:Sp_JointlyLipschitz}
Let $F: W^{1,2}(\Omega)\to \R$ be continuously differentiable
with  Lipschitz continuous derivative $F' : W^{1,2}(\Omega) \to W^{-1,2}(\Omega)$
and $\sup\{ \|F'(S(u))\|_{W^{-1,s}} : u \in \feas\} < \infty$.
For all $q \in \big[\frac{2s}{s-2},\infty\big)$, the mapping $(\feas,\|\cdot\|_{L^q}) \ni u \mapsto S_p(S(u), u) \in W^{1,2}_{\tilde\Gamma}(\Omega)$ 
is Lipschitz continuous with Lipschitz constant depending on $q$, $s$, $u_\ell$, and the data.  
\end{corollary}
\begin{proof}
For $u_1$, $u_2\in\feas$, we insert the zero
$0 = S_p(S(u_1),u_2) - S_p(S(u_1),u_2)$
and combine the Lipschitz continuity of $S_p(y,\cdot)$ from \cref{lem:SpLipschitz}
and the Lipschitz continuities of $S_p(\cdot,u_2)$ and $S$ from \cref{lem:S_Lipschitz}
to obtain
$
\|S_p(S(u_1),u_1)-S_p(S(u_2),u_2)\|_{W^{1,2}}
\le
L\|u_1-u_2\|_{L^q},
$
with a constant $L>0$ depending on $q$, $s$, $u_\ell$, and the data.  
This proves the claim.
\end{proof}

\subsection{First Order Derivatives of \texorpdfstring{$S$}{S} and \texorpdfstring{$F \circ S$}{F o S}}
By Proposition \ref{prp-new:elementary_properties} we know that there exists $s \in (2,\infty)$ such that the control-to-state map $S: U_{ad} \to W^{1,s}(\Omega)$ is well-defined. Moreover, by Lemma \ref{lem:S_Lipschitz} we have that for all $q \in \big[\frac{2s}{s-2},\infty\big)$, \\
$
S : (\feas,\|\cdot\|_{L^{q}}) \to W^{1,2}(\Omega) $
is Lipschitz continuous. Next we prove that $S$ is of class $C^1$. \\

\begin{lemma}\label{lem:S_derivative}
Let $u \in \feas$ and assume that $q \in \big[\frac{2s}{s-2},\infty\big)$.
Then the control-to-state operator
$
S : (\feas, \|\cdot\|_{L^q}) \to W^{1,2}(\Omega)
$
 is Fr\'echet differentiable at $u$
w.r.t.\ feasible variations. Moreover, given $u \in U_{ad}$ and $\phi \in L^q(\Omega)$ s.t. $u+\phi \in U_{ad}$, we have that $S'(u)[\phi]$ is the unique solution to
\begin{equation}\label{eq:def_z_u_phi}
\int_\Omega u \nabla z \cdot \nabla v \, dx
=
-\int_\Omega \phi \nabla y \cdot \nabla v \, dx, \quad\text{ for all }v \in W^{1,2}_{\tilde\Gamma}(\Omega).
\end{equation}
\end{lemma}

\begin{proof}
Let $u \in \feas$, and let $\phi \in L^q(\Omega)$ such that $u+\phi \in \feas$. 
The bilinear form
$
a(z,v)
:=
\int_\Omega u \nabla z \cdot \nabla v\,dx $
is continuous on
$W^{1,2}_{\tilde\Gamma}(\Omega)\times
W^{1,2}_{\tilde\Gamma}(\Omega)$ since
$u\in L^\infty(\Omega)$, and coercive because
$u\ge u_\ell>0$ almost everywhere in $\Omega$.
Let $G_\phi: W^{1,2}_{\tilde\Gamma}(\Omega) \to \R$ be given by $ \langle G_\phi, v\rangle
:=
-\int_{\Omega} \phi \nabla y \cdot \nabla v dx$. 
Using
H\"older's Inequality with exponents $q$, $t$ and $2$ such that 
$
\frac1q+\frac1t+\frac12=1,$  we obtain (due to the fact that $t \in (2, s]$) that
$$
\left|
\langle G_\phi, v\rangle
\right|
\le 
\|\phi\|_{L^q}
\|\nabla y\|_{L^t}
\|\nabla v\|_{L^2}
\le C
\|\phi\|_{L^q}
\|\nabla y\|_{L^s}
\|\nabla v\|_{L^2},$$
which shows that $ G_\phi \in 
W^{-1,2}(\Omega)$. Hence, by Lax--Milgram Theorem,
there exists a unique solution 
$z_{u, \phi} \in W^{1,2}_{\tilde\Gamma}(\Omega)$ to \eqref{eq:def_z_u_phi}. 
Moreover, we have that $ G_\phi \in W^{-1,s}(\Omega)$. Indeed,
using H\"older's inequality together with the $W^{1,s}$-regularity of $y$
and the $L^\infty$-regularity of $\phi$, we obtain
\[
|\langle G_\phi,v\rangle|
\le
\|\phi\|_{L^\infty}
\|\nabla y\|_{L^s}
\|\nabla v\|_{L^{s'}}
\le
\|\phi\|_{L^\infty}
\|y\|_{W^{1,s}}
\|v\|_{W^{1,s'}}.
\]
Thus Proposition~\ref{prp-new:elementary_properties} implies that
$ z_{u, \phi} \in W^{1,s}_{\tilde\Gamma}(\Omega),$  and there exists $C_1>0$ such that
$
\| z_{u, \phi}\|_{W^{1,s}}
\le
C_1
\|G_\phi\|_{W^{-1,s}}
\le
C_1
\|\phi\|_{L^\infty}
\|y\|_{W^{1,s}}.$
Therefore, the unique solution
$z_{u, \phi} \in W^{1,2}_{\tilde\Gamma}(\Omega)$ satisfies
$$
\|z_{u, \phi}\|_{W^{1,2}(\Omega)}
\le
C_2
\|\phi\|_{L^q}
\|y\|_{W^{1,s}}.$$
Now for all $v \in W_{\tilde\Gamma}^{1,2}(\Omega)$, the homogenized states $\tilde y = y - \tilde f_D$ and  $\tilde y_{\phi} = y_{\phi} - \tilde f_D$ satisfy
\begin{align}
\int_\Omega u \nabla \tilde y \cdot \nabla v \, dx
&=
\int_\Omega f v \, dx
-
\int_\Omega u \nabla \tilde f_D \cdot \nabla v \, dx
\label{eq:state_u}\\
\int_\Omega (u+\phi) \nabla \tilde y_{\phi} \cdot \nabla v \, dx
&=
\int_\Omega f v \, dx
-
\int_\Omega (u+\phi) \nabla \tilde f_D \cdot \nabla v \, dx
\label{eq:state_uh}
\end{align}
Subtracting \eqref{eq:state_u} from \eqref{eq:state_uh} yields the identity
\[
\int_\Omega u \nabla(\tilde y_{\phi} - \tilde y) \cdot \nabla v \, dx
+
\int_\Omega \phi \nabla \tilde y_{\phi} \cdot \nabla v \, dx
=
-  \int_\Omega \phi \nabla \tilde f_D \cdot \nabla v \, dx.
\]
Subtracting $\int_\Omega \phi \nabla \tilde y_{\phi} \cdot \nabla v \, dx$ from both sides and 
using $\tilde y_{\phi} + \tilde f_D = y_{\phi}$, the identity becomes
\begin{equation}\label{varformtildeuphi}
\int_\Omega u \nabla(\tilde{y}_{\phi} - \tilde{y}) \cdot \nabla v \, dx
=
-  \int_\Omega \phi \nabla y_{\phi} \cdot \nabla v \, dx.
\end{equation}
Combining \eqref{varformtildeuphi} with \eqref{eq:def_z_u_phi} we obtain 
\begin{equation}\label{varformderivative}
\int_\Omega u \nabla(\tilde{y}_{\phi} - \tilde{y} - z_{u, \phi}) \cdot \nabla v \, dx
=
-  \int_\Omega \phi \nabla (y_{\phi} -  y) \cdot \nabla v \, dx, \ \forall v \in W^{1,2}_{\tilde \Gamma}(\Omega).
\end{equation}
From \eqref{varformderivative}, using the lower bound on $u$ and the fact that $
y_\phi-y=\tilde y_\phi-\tilde y$, we obtain 
\begin{align}
u_{\ell} &  \|y_{\phi} - y - z_{u, \phi}\|^2_{W^{1,2}}
\leq  C_p \Big | \int_{\Omega} u  \nabla ( \tilde{y}_{\phi} - \tilde{y} -  z_{u, \phi} ) \nabla ( \tilde{y}_{\phi} - \tilde{y} -  z_{u, \phi} ) dx \Big | \nonumber \\
&= C_p \Big | \int_{\Omega} \phi   \nabla ( \tilde{y}_{\phi} - \tilde{y}  ) \nabla ( \tilde{y}_{\phi} - \tilde{y} - z_{u, \phi} ) dx \Big |  \leq C_p\| \phi\|_{L^q}
\|  \nabla ( \tilde{y}_{\phi} - \tilde{y}  )\|_{L^{t}}
\| \nabla ( \tilde{y}_{\phi} - \tilde{y} -  z_{u, \phi} )\|_{L^{2}}, \nonumber \\
\end{align}
where $t \in (2,s)$ is such that 
$
\frac1q+\frac1t+\frac12=1$. Then using interpolation between $L^2$ and $L^s$, there exists
$\theta=\frac{2(s-t)}{t(s-2)} \in(0,1)$ such that 
\begin{align}
u_{\ell} &  \|y_{\phi} - y -  z_{u, \phi}\|^2_{W^{1,2}}
\le C_p\|\phi\|_{L^q}
\|\nabla(y_{\phi} - y)\|_{L^2}^{\theta}
\|\nabla(y_{\phi} - y)\|_{L^s}^{1-\theta}
\| \nabla ( \tilde{y}_{\phi} - \tilde{y} -  z_{u, \phi} )\|_{L^{2}}.
\end{align}
Finally, using the the Lipschitz continuity of the control-to-state operator
from Lemma \ref{lem:S_Lipschitz} we obtain
\begin{align}
u_{\ell} &  \|y_{\phi} - y - z_{\phi}\|^2_{W^{1,2}}
\leq C \| \phi\|_{L^q}^{1+\theta}
C(s)(\|f\|_{W^{-1,s}} + \|f_D\|_{W^{1-1/s,s}(\tilde{\Gamma})})^{1 - \theta}
\|y_{\phi} - y - z_{\phi}\|_{W^{1,2}},\nonumber
\end{align}
for real constant $C>0$, which allows us to conclude that
\[
\frac{
\|y_{\phi} - y - z_{\phi}\|_{W^{1,2}}
}{
\|\phi\|_{L^q}
}
\to 0,
\qquad
\text{as }
\|\phi\|_{L^q}\to 0.
\]
Thus $S$ is Fr\'echet differentiable at $u$ and 
$S'(u)[\phi]= z_{u, \phi}$. 

\end{proof}

\begin{theorem}\label{thm:old_regularity_of_first_derivative}
Let $q \in \big(\frac{2s}{s-2},\infty\big)$, and $F: W^{1,2}(\Omega)\to \R$ be continuously differentiable with  Lipschitz
continuous derivative $F' : W^{1,2}(\Omega) \to W^{-1,2}(\Omega)$ and $\sup\{ \|F'(S(u))\|_{W^{-1,s}} : u \in \feas\} < \infty$.
Then the mapping $F \circ S : (\feas, \|\cdot\|_{L^q}) \to \R$ is 
Fr\'echet differentiable at $u \in \feas$ with respect to feasible variations. Specifically, with $y = S(u)$ and $p = S_p(y, u)$, we have
\[ \underset{\substack{\|\phi\|_{L^q} \to 0\\ u+\phi\in\feas}}\lim
\frac{|F(S(u + \phi)) - F(y) - (\phi \nabla y, \nabla p)_{L^2}|}{\|\phi\|_{L^q}} = 0
\]
Moreover, the mapping
$
\feas \ni u
\mapsto
\nabla y(u)\cdot \nabla p(u)
\in L^{\frac{2s}{s+2}}(\Omega)
$
is Lipschitz continuous with respect to the $L^q(\Omega)$-norm.
\end{theorem}
\begin{proof}
Using Theorem \ref{lem:S_derivative} and the assumptions on  $F$ we obtain that $(F \circ S)'(u)[\phi] = \langle F'(y),z_{u,\phi}\rangle$, for $u \in \feas$, and $\phi \in L^q(\Omega)$ such that $u+\phi \in \feas$. 

Now we use
the adjoint equation \eqref{eq:adjoint} (tested with $z_{u, \phi}$) and the variational form \eqref{eq:def_z_u_phi} of
$z_{u, \phi}$ (tested with $p$) to deduce that
$ \langle F'(y), z_{u, \phi}\rangle
 = - \int_{\Omega} u \nabla p \cdot \nabla z_{u, \phi}
 = \int_{\Omega} \phi \nabla y \cdot \nabla p.$ 
Now, let $u_1$, $u_2 \in \feas$ and set
$y_1 \coloneqq S(u_1)$, $y_2 \coloneqq S(u_2)$,
$p_1 \coloneqq S_p(y_1,u_1)$, and
$p_2 \coloneqq S_p(y_2,u_2)$.
By H\"older's inequality with
$
\frac{s+2}{2s}=\frac12+\frac1s,$ 
we obtain
\begin{align*}
\|\nabla y_1 \cdot \nabla p_1 - \nabla y_2 \cdot \nabla p_2\|_{L^{\frac{2s}{s+2}}}
&\le
\|(\nabla y_1-\nabla y_2)\cdot \nabla p_1\|_{L^{\frac{2s}{s+2}}}
+
\|\nabla y_2\cdot(\nabla p_1-\nabla p_2)\|_{L^{\frac{2s}{s+2}}} \\
&\le
\|\nabla y_1-\nabla y_2\|_{L^2}
\|\nabla p_1\|_{L^s}
+
\|\nabla y_2\|_{L^s}
\|\nabla p_1-\nabla p_2\|_{L^2}.
\end{align*}
By Proposition~\ref{prp-new:elementary_properties}, there exists a constant $C_y>0$,
depending only on the data, the domain, and the admissible bounds
$u_\ell,u_u$, such that
$
\|y(u)\|_{W^{1,s}(\Omega)}
\le C_y,
\; \forall u\in\feas .
$
Similarly, the assumption
$
\sup_{u\in\feas}\|F'(S(u))\|_{W^{-1,s}(\Omega)}<\infty$
is used to obtain a uniform $W^{1,s}$-bound for the adjoint variables.
Note that $p(u)$ solves
$
A(u)p(u)=-F'(S(u)),$ 
the $W^{1,s}$-regularity estimate for the adjoint equation yields
$
\|p(u)\|_{W^{1,s}(\Omega)}
\le
C\|F'(S(u))\|_{W^{-1,s}(\Omega)}$ 
for all $u\in\feas$. Hence, there exists a constant $C_p>0$ such that
$
\|p(u)\|_{W^{1,s}(\Omega)}
\le C_p
\; \forall u\in\feas .
$
Hence, $
\|\nabla p_1\|_{L^s}\le C_p,
\; 
\|\nabla y_2\|_{L^s}\le C_y .$ Combining these uniform bounds with
\cref{lem:S_Lipschitz,cor:Sp_JointlyLipschitz}, we obtain
\[
\|\nabla y_1-\nabla y_2\|_{L^2}
+
\|\nabla p_1-\nabla p_2\|_{L^2}
\le C\|u_1-u_2\|_{L^q}.
\]
Therefore,
\[
\|\nabla y_1 \cdot \nabla p_1 - \nabla y_2 \cdot \nabla p_2\|_{L^{\frac{2s}{s+2}}}
\le C\|u_1-u_2\|_{L^q},
\]
where the constant $C>0$ is independent of $u_1,u_2\in\feas$.
This proves the final claim. 

\end{proof}
\begin{remark}\label{rem:jprime}  Let $
j:\feas \to \mathbb{R},$ be defined by  $j(u) = (F \circ S) (u).$ 
	For all $u \in \feas$ and $\phi$ that satisfy the assumptions of \cref{thm:old_regularity_of_first_derivative}, with $y = S(u),$ we write $ j'(u)[\phi] \coloneqq \int_{\Omega} \phi \nabla p \cdot \nabla y\,dx $
	and observe that this is a bounded linear operator in $\phi$ with respect to the $L^q$-norm.  Moreover, the Lipschitz continuity of the mapping
$
\feas\ni u
\mapsto
\nabla y(u)\cdot \nabla p(u)
\in L^{\frac{2s}{s+2}}(\Omega)
$
yields corresponding continuity properties of the derivative
$j'(u)$ . These properties
will be used later in Section~3 for the optimality system and in
Section~4 in the trust-region convergence analysis.
\end{remark}

\subsection{Second Order Derivatives of \texorpdfstring{$S$}{S} and \texorpdfstring{$F \circ S$}{F o S}} 

From the previous section we have that for  $u \in U_{ad}$ and $\phi \in L^q(\Omega)$ such that $u+\phi \in U_{ad}$, $z_{u, \phi} = S'(u)[\phi] \in W^{1,2}_{\tilde\Gamma}(\Omega)$ is the unique solution to \eqref{eq:def_z_u_phi}. Now we obtain the following property for $z_{u, \phi}$, which will be used in the proof of the second-order
differentiability of the control-to-state operator. 

\begin{lemma}\label{lem:Sprime_lipschitz}
Let $q \in \big[\frac{2s}{s-2},\infty\big)$ and let $\phi\in L^\infty(\Omega)$.
For every $u\in \feas$, let $y(u)=S(u)\in W^{1,2}(\Omega)$ denote the state.
Then there exists a constant $L>0$ such that for all
$u_1,u_2\in \feas$ we have
\[
\|z_{u_1,\phi}-z_{u_2,\phi}\|_{W^{1,2}(\Omega)}
\le
L\max\{1,\|\phi\|_{L^\infty}\}
\|u_1-u_2\|_{L^q(\Omega)}.
\]
\end{lemma}
\begin{proof}
 Let $u_1$, $u_2\in \feas$ be arbitrary and set
$ y_i \coloneqq S(u_i),\; z_i \coloneqq z_{u_i,\phi} = S'(u_i)[\phi]
$
for $i \in \{1,2\}$. By definition, $z_i$ satisfies \eqref{eq:def_z_u_phi} with $u=u_i$ and $y(u_i)=y_i$.
We therefore subtract \eqref{eq:def_z_u_phi} for $i = 2$ from \eqref{eq:def_z_u_phi} for $i = 1$
for arbitrary test functions $v\in W^{1,2}_{\tilde\Gamma}(\Omega)$ and obtain
\begin{equation*}
\int_\Omega u_1 \nabla z_1\cdot\nabla v\,dx
-\int_\Omega u_2 \nabla z_2\cdot\nabla v\,dx
=
-\int_\Omega \phi \nabla (y_1-y_2)\cdot\nabla v\,dx.
\end{equation*}
Subtracting $\int_\Omega u_1\nabla z_2\cdot\nabla v\,dx$ from both sides and reorganizing terms yields
\begin{equation*}
\int_\Omega u_1\nabla(z_1-z_2)\cdot\nabla v\,dx
=
-\int_\Omega \phi\,\nabla(y_1-y_2)\cdot\nabla v\,dx
-\int_\Omega (u_1-u_2)\nabla z_2\cdot\nabla v\,dx.
\end{equation*}
The choice $v=z_1-z_2\in W^{1,2}_{\tilde\Gamma}(\Omega)$ then yields
\begin{align*}
\int_\Omega u_1|\nabla(z_1-z_2)|^2\,dx
&=
-\int_\Omega \phi\,\nabla(y_1-y_2)\cdot\nabla(z_1-z_2)\,dx
-\int_\Omega (u_1-u_2)\nabla z_2\cdot\nabla(z_1-z_2)\,dx.
\end{align*}
Using the fact that $u_1\ge u_\ell>0$ a.e.\ in $\Omega,$ together with similar estimates
as in the proof of \cref{lem:S_Lipschitz}, we obtain
\begin{equation} \label{inequality}
u_\ell\|\nabla(z_1-z_2)\|_{L^2(\Omega,\R^d)}
\le 
\|\phi\|_{L^\infty(\Omega)} \|\nabla(y_1-y_2)\|_{L^2(\Omega,\R^d)}
+ \|\nabla z_2\|_{L^{t}(\Omega,\R^d)}\|u_1-u_2\|_{L^q(\Omega)}
\end{equation}
with $t \coloneqq 2q/(q-2)$ so that $1/q+1/t+1/2=1.$ Since \(q\ge \frac{2s}{s-2}\), we have \(t\le s\). Hence, by H\"older's inequality and the continuous embedding \(L^s(\Omega)\hookrightarrow L^t(\Omega)\),  $
\|\nabla z_2\|_{L^t(\Omega,\mathbb{R}^d)}
\le C\|\nabla z_2\|_{L^s(\Omega,\mathbb{R}^d)}.$ 
Using \cref{lem:S_Lipschitz}, we obtain that $
\|\nabla(y_1-y_2)\|_{L^2(\Omega,\mathbb{R}^d)}
\le L_S\|u_1-u_2\|_{L^q(\Omega)}.$ 
Therefore, inequality \eqref{inequality} yields
\[
u_\ell\|\nabla(z_1-z_2)\|_{L^2(\Omega,\mathbb{R}^d)}
\le
\Big(
L_S\|\phi\|_{L^\infty(\Omega)}
+
C\|\nabla z_2\|_{L^s(\Omega,\mathbb{R}^d)}
\Big)
\|u_1-u_2\|_{L^q(\Omega)}.
\]
It remains to uniformly bound $\|\nabla z_2\|_{L^s(\Omega,\mathbb{R}^d)}$ in terms of $\| \phi \|_{L^{\infty}}.$ Define the functional $G\in W^{-1,s}(\Omega)$ by
\[
\langle G,v\rangle_{W^{-1,s},W^{1,s'}}
:=
-\int_\Omega \phi \nabla y_2\cdot \nabla v\,dx
\qquad
\forall v\in W^{1,s'}_{\tilde\Gamma}(\Omega).
\]
Then
\begin{align*}
\|G\|_{W^{-1,s}(\Omega)}
&=
\sup_{\|v\|_{W^{1,s'}(\Omega)}\le1}
\left|
\langle G,v\rangle_{W^{-1,s},W^{1,s'}}
\right|  \le
\|\phi\|_{L^\infty(\Omega)}
\|\nabla y_2\|_{L^s(\Omega,\mathbb R^d)} .
\end{align*}
Consequently, the same regularity theory as for the state equation
applies and we
deduce from \cref{prp-new:elementary_properties}
and the assumptions on $\phi$ that
\[ \|z_2\|_{W^{1,s}} \le  K_2(s) 
\|\phi\|_{L^\infty(\Omega)}\| \nabla y_2\|_{L^s(\Omega,\R^d)}
\le  K_2(s)\|\phi\|_{L^\infty}C_{y}
\]
with $C_y \coloneqq \|\tilde{f}_D\|_{W^{1,s}}
+ K_2(s)\bigl(\|f\|_{W^{-1,s}} 
+ u_u\|\nabla \tilde{f}_D\|_{L^s}\bigr)$.
\end{proof}
To prove second-order differentiability of the control-to-state
operator, we later need to estimate terms of the form
$
\phi \nabla z_{u,\psi}$
in \(L^2(\Omega)\). For this purpose, the \(W^{1,2}\)-regularity of
\(z_{u,\psi}\) is not sufficient, and we require the following improved
\(W^{1,\sigma}\)-regularity result for first variations. This is presented in the next lemma.

\begin{lemma}\label{lem:firstvar-Lqtilde}
Let \(u\in U_{\rm ad}\) and let \(y=S(u)\).
Let \(\hat q \in(1,\infty)\) and define  $\sigma\in(2,s)$ by
$1/s+1/{\hat q}=1/\sigma,$ with $
\sigma' \coloneqq \frac{\sigma}{\sigma-1}$.
Then, for every \(\phi\in L^{\hat q}(\Omega)\), there exists a unique
$
z_{u,\phi}=S'(u)[\phi]
\in
W^{1,\sigma}_{\tilde\Gamma}(\Omega)
$
satisfying
\[
\int_\Omega
u\nabla z_{u,\phi}\cdot\nabla v\,dx
=
-\int_\Omega
\phi\nabla y\cdot\nabla v\,dx
\qquad
\forall v\in W^{1,\sigma'}_{\tilde\Gamma}(\Omega).
\]
Moreover, there exists a constant \(C>0\), depending only on
\(\Omega\), \(\sigma\), the admissible bounds
\(u_\ell,u_u\), and the data \(f,f_D\), such that
$
\|z_{u,\phi}\|_{W^{1,\sigma}(\Omega)}
\le
C\|\phi\|_{L^{\hat q}(\Omega)}.
$
\end{lemma}
\begin{proof}
First note that
$
\hat q
=
\frac{s\sigma}{s-\sigma}.
$
Since \(\sigma\in(2,s)\), we obtain
$
\hat q
>
\frac{2s}{s-2}.$ 
Hence, the requirement from
\cref{thm:old_regularity_of_first_derivative}
is satisfied. Let \(\phi\in L^{\hat q}(\Omega)\).
Since \(y=S(u)\in W^{1,s}(\Omega)\), we have that $
\nabla y\in L^s(\Omega,\mathbb R^d).$ 
By the choice of \(\hat q\), $1/s+1/{\hat q}=1/\sigma$, 
and therefore Hölder's inequality yields
$\phi\nabla y\in L^\sigma(\Omega,\mathbb R^d)$.

We now define the functional
$
G_1\in W^{-1,\sigma}(\Omega)$ 
by
\[
\langle G_1,v\rangle_{
W^{-1,\sigma},
W^{1,\sigma'}_{\tilde\Gamma}}
:=
-\int_\Omega
\phi\nabla y\cdot\nabla v\,dx
\qquad
\forall
v\in W^{1,\sigma'}_{\tilde\Gamma}(\Omega).
\]
Then, for all
\(v\in W^{1,\sigma'}_{\tilde\Gamma}(\Omega)\),
$
|\langle G_1,v\rangle|
\le
\|\phi\nabla y\|_{L^\sigma(\Omega,\mathbb R^d)}
\|\nabla v\|_{L^{\sigma'}(\Omega,\mathbb R^d)}
$
and hence
\[
|\langle G_1,v\rangle|
\le
\|\phi\|_{L^{\hat q}(\Omega)}
\|\nabla y\|_{L^s(\Omega,\mathbb R^d)}
\|v\|_{W^{1,\sigma'}(\Omega)}.
\]
Consequently,
$
\|G_1\|_{W^{-1,\sigma}(\Omega)}
\le
C
\|\phi\|_{L^{\hat q}(\Omega)}
\|\nabla y\|_{L^s(\Omega,\mathbb R^d)}.$ Since \(\sigma\in(2,s)\), Proposition \ref{prop:Fu_W-1-s} applies.
Therefore, there exists a unique
$
z_{u,\phi}\in
W^{1,\sigma}_{\tilde\Gamma}(\Omega)$ 
such that
\[
\int_\Omega
u\nabla z_{u,\phi}\cdot\nabla v\,dx
=
\langle G_1,v\rangle
\qquad
\forall
v\in W^{1,\sigma'}_{\tilde\Gamma}(\Omega),
\]
and
$
\|z_{u,\phi}\|_{W^{1,\sigma}(\Omega)}
\le
K_2(\sigma)
\|G_1\|_{W^{-1,\sigma}(\Omega)}.$
Since
$
\|y\|_{W^{1,s}(\Omega)}
\le
C_y,
\;
\forall u\in\feas,$ 
it follows that
$
\|z_{u,\phi}\|_{W^{1,\sigma}(\Omega)}
\le
C
\|\phi\|_{L^{\hat q}(\Omega)}.
$
This proves the claim.
\end{proof}
Before continuing with our argument, we supply an auxiliary lemma
on the relationship of integrability and interpolation indices
required below.

\begin{lemma}\label{lem:S_second_derivative}
 Let $\tilde q > \frac{4s}{s-2}$.  
Let $u \in \feas$. Then the map
$\displaystyle
S' :
(\feas,\|\cdot\|_{L^{\tilde q}(\Omega)})
\to
\mathcal L(L^{\tilde q}(\Omega), W^{1,2}(\Omega))
$
is Fr\'echet differentiable at $u$ with respect to feasible variations. 
Specifically, let $\phi$, $\psi$ satisfy $u+\phi$, $u + \psi \in \feas$ and recall that $z_{u, \phi} = S'(u)[\phi]$ and
$z_{u, \psi} = S'(u)[\psi]$ are the corresponding unique solutions to \eqref{eq:def_z_u_phi}. Let
$z_{u,\phi,\psi} \in W^{1,2}_{\tilde\Gamma}(\Omega)$ be the unique solution to
\begin{equation}\label{eq:def_z_phipsi}
\int_\Omega u \nabla z_{u,\phi,\psi} \cdot \nabla v \, dx
=
-\int_\Omega \psi \nabla z_{\phi} \cdot \nabla v \, dx
-
\int_\Omega \phi \nabla z_{\psi} \cdot \nabla v \, dx,
\qquad
\forall v \in W^{1,2}_{\tilde\Gamma}(\Omega).
\end{equation}
Then the second derivative of $S$ at $u$ is
$
S''(u)[\phi,\psi] = z_{u,\phi,\psi},$
and

\[
\underset{\substack{\|\phi\|_{L^{\tilde q}} \to 0\\ u+\phi\in\feas}}\lim
\;
\sup_{\psi \in L^{\tilde q}(\Omega)\setminus\{0\}}
\frac{
\| S'(u+\phi)[\psi] - S'(u)[\psi] - z_{u,\phi,\psi} \|_{W^{1,2}}
}{
\|\phi\|_{L^{\tilde q}}\,
\|\psi\|_{L^{\tilde q}}
}
= 0.
\]
In addition, there exists $C > 0$ such that $\|z_{u,\phi,\psi}\|_{W^{1,2} _{\tilde\Gamma}(\Omega)} \le C \|\phi\|_{L^{q_2}}\|\psi\|_{L^{q_2}}$
for $q_2 = 4s/(s - 2)$.

\end{lemma}

\begin{proof}
We set $y \coloneqq S(u)$, $y_\phi \coloneqq S(u+\phi)$ and recall the homogenized states
$\tilde y \coloneqq y-\tilde f_D$ and $\tilde y_\phi \coloneqq y_\phi -\tilde f_D$.
Moreover, we know that $z_{u, \psi}$ and $z_{u, \psi}[\phi]$ solve
\begin{align}
\int_\Omega u\,\nabla z_{u, \psi}\cdot \nabla v\,dx
&=
-\int_\Omega \psi\,\nabla y\cdot \nabla v\,dx,
& \forall v\in W^{1,2}_{\tilde\Gamma}(\Omega), \label{eq:zpsi_u_Sprimeproof}\\
\int_\Omega (u+\phi)\,\nabla z_{u, \psi}[\phi]\cdot \nabla v\,dx
&=
-\int_\Omega \psi\,\nabla y_\phi\cdot \nabla v\,dx,
&\forall v\in W^{1,2}_{\tilde\Gamma}(\Omega). \label{eq:zpsi_uphi_Sprimeproof}
\end{align}
We define the remainder
$
r_\phi \coloneqq z_{u, \psi}[\phi]-z_{u, \psi}-z_{u,\phi,\psi}.
$
Thus subtracting \eqref{eq:def_z_phipsi} and \eqref{eq:zpsi_u_Sprimeproof} from \eqref{eq:zpsi_uphi_Sprimeproof} yields
\begin{equation}\label{eq:rphi_identity_Sprimeproof}
\int_\Omega u\,\nabla r_\phi\cdot\nabla v\,dx
=
-\int_\Omega \psi\,\nabla\big((y_\phi-y)-z_{u, \phi}\big)\cdot\nabla v\,dx
-\int_\Omega \phi\,\nabla\big(z_{u, \psi}[\phi]-z_{u, \psi}\big)\cdot\nabla v\,dx .
\end{equation}
Testing with $v=r_\phi$, using coercivity $u\ge u_\ell>0$ and Poincar\'e's inequality, we obtain
\begin{equation}\label{eq:rphi_est_Sprimeproof}
\|r_\phi\|_{W^{1,2}}
\le
C\Big(
\|\psi\,\nabla((y_\phi-y)-z_{u, \phi})\|_{L^2}
+
\|\phi\,\nabla(z_{u, \psi}[\phi]-z_{u, \psi})\|_{L^2}
\Big)
\end{equation}
for some $C > 0$. Throughout the remainder of the proof, $C>0$ denotes a generic constant that may change from line to line but is independent of the perturbations $\phi$ and $\psi$. We estimate the two terms on the right-hand side of \eqref{eq:rphi_est_Sprimeproof} separately.

 Beginning with the first one, we let $\tilde t$ and $\sigma$
satisfy $ \frac{1}{\tilde t}+\frac{1}{\tilde q}=\frac12,\; \; 
\frac{1}{\sigma}=\frac1s+\frac{1}{\tilde q}. $ By definition, $\tilde t>2$ and $\sigma<s$. Moreover, since
$\tilde q>\frac{4s}{s-2}$, we have
$\frac{1}{\tilde t}-\frac{1}{\sigma}
=\frac12-\frac1s-\frac2{\tilde q}>0. $ Hence, $2<\tilde t<\sigma<s$, and $\theta \coloneqq
\frac{\frac{1}{\tilde t}-\frac{1}{\sigma}}
{\frac12-\frac{1}{\sigma}}\in(0,1) $
satisfies $
\frac1{\tilde t}
=\frac{\theta}{2}+\frac{1-\theta}{\sigma}.$  For the first term on the right-hand side of \eqref{eq:rphi_est_Sprimeproof},
we set
$d_\phi \coloneqq y_{\phi}-y-z_{u,\phi}
\in W^{1,2}_{\tilde\Gamma}(\Omega)$.
We interpolate between $L^2(\Omega)$ and
$L^\sigma(\Omega)$ and obtain
\begin{equation}\label{eq:interpolation_en}
\begin{aligned}
\|\psi \nabla d_\phi\|_{L^2}
&\le \|\psi\|_{L^{\tilde q}}
     \|\nabla d_\phi\|_{L^{\tilde t}} \le \|\psi\|_{L^{\tilde q}}
     \|\nabla d_\phi\|_{L^2}^{\theta}
     \|\nabla d_\phi\|_{L^\sigma}^{1-\theta}.
\end{aligned}
\end{equation}
We first estimate $y_\phi-y$ in $W^{1,\sigma}(\Omega)$.
Subtracting the state equations gives
\[
\int_\Omega u\nabla(y_\phi-y)\cdot\nabla v\,dx
=
-\int_\Omega \phi\nabla y_\phi\cdot\nabla v\,dx.
\]
Since $1/\sigma=1/s+1/\tilde q$,
H\"older's inequality and
\cref{prp-new:elementary_properties} imply
$$\|y_\phi-y\|_{W^{1,\sigma}} \le C\|\phi\|_{L^{\tilde q}}  \|y_\phi\|_{W^{1,s}} \le C\|\phi\|_{L^{\tilde q}}.$$
Since $1/\sigma=1/s+1/\tilde q$ and $\sigma\in(2,s)$, we can apply \cref{lem:firstvar-Lqtilde} with
$\hat q=\tilde q$. Hence, $\|z_{u,\phi}\|_{W^{1,\sigma}} \le C\|\phi\|_{L^{\tilde q}},$ and so 
$$ \|d_\phi\|_{W^{1,\sigma}} \le C\|\phi\|_{L^{\tilde q}}. $$
On the other hand, from
\eqref{varformderivative}, using the lower
bound on $u$, H\"older's inequality,
and $\tilde t<\sigma$, we obtain
\begin{equation}\label{eq:lemma311_L2_bound}
\begin{aligned}
\|d_\phi\|_{W^{1,2}}
&\le C\|\phi\|_{L^{\tilde q}}
       \|\nabla(y_\phi-y)\|_{L^{\tilde t}} \le C\|\phi\|_{L^{\tilde q}}^2.
\end{aligned}
\end{equation}
Combining these estimates with
\eqref{eq:interpolation_en}, we obtain
\[
\|\psi\nabla d_\phi\|_{L^2}
\le
C\|\psi\|_{L^{\tilde q}}
\|\phi\|_{L^{\tilde q}}^{1+\theta}.
\]
We thus obtain, for some $C>0$,
\begin{equation}\label{eq:first_term_final}
\sup_{\psi\in L^{\tilde q}(\Omega)\setminus\{0\}}
\frac{
\|\psi\nabla(y_\phi-y-z_{u,\phi})\|_{L^2}
}{
\|\phi\|_{L^{\tilde q}}\|\psi\|_{L^{\tilde q}}
}
\le C\|\phi\|_{L^{\tilde q}}^\theta.
\end{equation}

We now estimate the second term in \eqref{eq:rphi_est_Sprimeproof}, using the same exponents $\tilde t$, $\sigma$, and $\theta$
defined above. Set
$ \delta z_\psi \coloneqq z_{u,\psi}[\phi]-z_{u,\psi}.$
By \cref{lem:firstvar-Lqtilde},
we obtain
$$ \|z_{u,\psi}[\phi]\|_{W^{1,\sigma}} + \|z_{u,\psi}\|_{W^{1,\sigma}} \le C\|\psi\|_{L^{\tilde q}}. $$
Then,
$ \|\delta z_\psi\|_{W^{1,\sigma}} \le C\|\psi\|_{L^{\tilde q}}. $ Subtracting \eqref{eq:zpsi_u_Sprimeproof}
from \eqref{eq:zpsi_uphi_Sprimeproof},
we obtain
\begin{equation*}
\int_\Omega u\nabla\delta z_\psi \cdot\nabla v\,dx
= -\int_\Omega \psi\nabla(y_\phi-y)\cdot\nabla v\,dx 
-\int_\Omega \phi\nabla z_{u,\psi}[\phi]\cdot\nabla v\,dx.
\end{equation*}
Testing with $v=\delta z_\psi$,
using the lower bound on $u$,
H\"older's inequality, and
$\tilde t<\sigma$, we obtain
\[
\begin{aligned}
\|\delta z_\psi\|_{W^{1,2}}
&\le C\Big(
\|\psi\|_{L^{\tilde q}}
\|y_\phi-y\|_{W^{1,\sigma}} +
\|\phi\|_{L^{\tilde q}}
\|z_{u,\psi}[\phi]\|_{W^{1,\sigma}}
\Big)\\
&\le C
\|\phi\|_{L^{\tilde q}}
\|\psi\|_{L^{\tilde q}}.
\end{aligned}
\]

By H\"older's inequality and interpolation,
we then obtain
\begin{align*}
\|\phi\nabla\delta z_\psi\|_{L^2}
&\le
\|\phi\|_{L^{\tilde q}}
\|\nabla\delta z_\psi\|_{L^{\tilde t}}
\le
\|\phi\|_{L^{\tilde q}}
\|\nabla\delta z_\psi\|_{L^2}^{\theta}
\|\nabla\delta z_\psi\|_{L^\sigma}^{1-\theta}.
\end{align*}
Combining the above estimates yields
\begin{equation}\label{eq:second_term_final}
\|\phi\nabla(z_{u,\psi}[\phi]-z_{u,\psi})\|_{L^2}
\le
C\|\phi\|_{L^{\tilde q}}^{1+\theta}
\|\psi\|_{L^{\tilde q}}.
\end{equation}

Inserting \eqref{eq:first_term_final} and \eqref{eq:second_term_final} into \eqref{eq:rphi_est_Sprimeproof}, we obtain
$
\frac{\|r_\phi\|_{W^{1,2}}}
{\|\phi\|_{L^{\tilde q}}\|\psi\|_{L^{\tilde q}}}
\le C\|\phi\|_{L^{\tilde q}}^\theta$
for some $C>0$. Since $\theta>0$, the right-hand side
tends to zero whenever
$\|\phi\|_{L^{\tilde q}}\to0$.
This proves that
$
S':
(\feas,\|\cdot\|_{L^{\tilde q}(\Omega)})
\to
\mathcal L(L^{\tilde q}(\Omega),W^{1,2}(\Omega))
$
is Fr\'echet differentiable at $u$
with respect to feasible variations, with
$
S''(u)[\phi,\psi]=z_{u,\phi,\psi}.
$

It remains to prove the continuity estimate on the bilinear mapping $(\feas - u) \times (\feas - u) \ni (\phi,\psi) \mapsto z_{u,\phi,\psi} \in W^{1,2}_{\tilde\Gamma}(\Omega)$,
where the bilinearity follows directly from \eqref{eq:def_z_phipsi}. We test \eqref{eq:def_z_phipsi} with $v = z_{u,\phi,\psi}$, use $u \ge u_\ell > 0$, 
and choose $\sigma_0 \coloneqq \frac{4s}{s+2}$
and $q_2 \coloneqq \frac{4s}{s-2}$.  Then we obtain
\[ \|z_{u,\phi,\psi}\|_{W^{1,2}}
\le C\big( \|\psi\|_{L^{q_2}}
\|z_{u,\phi}\|_{W^{1,\sigma_0}}
+ \|\phi\|_{L^{q_2}}
\|z_{u,\psi}\|_{W^{1,\sigma_0}}
\big)
\]
for some $C > 0$ by Poincar\'{e}’s inequality on $W^{1,2}_{\tilde{\Gamma}}(\Omega)$, and Hölder’s inequality.
Then the claim follows from \cref{lem:firstvar-Lqtilde}.
\end{proof}

\begin{remark}
We highlight that the first-order differentiability of $S$
and $F\circ S$ is established for
$q \in \big(\frac{2s}{s-2},\infty\big)$,
whereas the second-order differentiability results
require the stronger condition
$\tilde q \in \big(\frac{4s}{s-2},\infty\big)$.
\end{remark}
The following theorem establishes
the second-order differentiability of the reduced objective
functional $j = F\circ S$.

\begin{theorem}\label{thm:second_order}
Let $\tilde q> \frac{4s}{s-2}$ . Let $F:W^{1,2}_{\tilde\Gamma}(\Omega)\to\mathbb R$
be twice continuously Fr\'echet differentiable and assume that
\[
F'(y)\in W^{-1,2}(\Omega),
\qquad
F''(y)\in \mathcal L(W^{1,2}_{\tilde\Gamma}(\Omega),W^{-1,2}(\Omega))
\qquad
\text{for all } y\in W^{1,2}_{\tilde\Gamma}(\Omega),
\]
with $F'$ and $F''$ Lipschitz continuous on $W^{1,2}_{\tilde\Gamma}(\Omega)$.

Then $j \coloneqq F\circ S:(\feas,\|\cdot\|_{L^{\tilde q}(\Omega)})\to\mathbb R$
is twice Fr\'echet differentiable on $\feas$ for feasible variations.
More precisely, for $u\in\feas$, let
\[
y\coloneqq S(u),\qquad p\coloneqq S_p(y,u),
\qquad z_{u, \phi} \coloneqq S'(u)[\phi],\qquad z_{u, \psi}\coloneqq S'(u)[\psi],
\qquad z_{u,\phi,\psi}\coloneqq S''(u)[\phi,\psi].
\]
Then, for all $\phi$, $\psi\in U - u$, we have
\begin{align} \label{j'}
j'(u)[\phi]
=
\langle F'(y),z_{u, \phi}\rangle_{W^{-1,2},W^{1,2}_{\tilde\Gamma}}
=
\int_\Omega \phi\,\nabla y\cdot\nabla p\,dx,
\end{align}
and
\begin{align} \label{j''}
j''(u)[\phi,\psi]
&=
\langle F''(y)z_{u, \psi},z_{u, \phi}\rangle_{W^{-1,2},W^{1,2}_{\tilde\Gamma}}
+
\int_\Omega \psi\,\nabla z_{u, \phi}\cdot\nabla p\,dx
+
\int_\Omega \phi\,\nabla z_{u, \psi}\cdot\nabla p\,dx .
\end{align}
Moreover,
\[
\lim_{\substack{\eta\to 0\\ u+\eta\in\feas}}
\frac{
\bigl|j(u+\eta)-j(u)-j'(u)[\eta]-\frac12\,j''(u)[\eta,\eta]\bigr|
}{
\|\eta\|_{L^{\tilde q}(\Omega)}^2
}
=0,
\]
and the bilinear map $j''(u):L^{\tilde q}(\Omega)\times L^{\tilde q}(\Omega)\to\mathbb \R$
is bounded.
\end{theorem}
\begin{proof}
By \cref{lem:S_second_derivative}, the control-to-state operator
$S:(\feas,\|\cdot\|_{L^{\tilde q}})\to W^{1,2}_{\tilde\Gamma}(\Omega)$
is twice Fréchet differentiable for feasible variations.

Due to the assumed twice continuous Fréchet differentiability of $F$ and the fact that
the remainder term estimate for the inner derivative holds for feasible variations,
the chain rule in Banach spaces may be applied to deduce that the reduced functional
$j$ is twice Fréchet differentiable for feasible variations.

Specifically, we obtain for $u \in \feas$ and $\phi \in \feas - u$ that
$ j'(u)[\phi] = \langle F'(y),S'(u)[\phi]\rangle = \langle F'(y),z_{u, \phi}\rangle
$
for the first derivative and, recalling $y = S(u)$,
\begin{align} \label{j''-proof}
    j''(u)[\phi,\psi]
= \langle F''(y)z_{u, \psi},z_{u, \phi}\rangle
+
\langle F'(y),S''(u)[\phi,\psi]\rangle = \langle F''(y)z_{u, \psi},z_{u, \phi}\rangle
+
\langle F'(y),z_{u,\phi,\psi}\rangle
\end{align}
for the second derivative, where $\langle \cdot,\cdot \rangle$ is the duality pairing
between $W^{1,2}_{\tilde\Gamma}(\Omega)$ and $W^{-1,2}(\Omega)$.

Recall that the adjoint state $p$ satisfies
\begin{align}\label{adjoint}
\int_\Omega u \nabla p \cdot \nabla v \, dx
=
- \langle F'(y), v \rangle
\qquad \text{for all } v \in W^{1,2}_{\tilde\Gamma}(\Omega)
\end{align}
Testing the above identity with $v = z_{u, \phi}$, and combining it with the variational formulation of $z_{u, \phi}$ from \eqref{eq:def_z_u_phi}, we obtain that
$
\langle F'(y), z_{u, \phi} \rangle
=
\int_\Omega \phi \, \nabla y \cdot \nabla p \, dx$
 which establishes \eqref{j'}.

To rewrite the second term in $j''(u)[\phi,\psi]$, we test \eqref{adjoint} with
$v=z_{u,\phi,\psi}$ and use the variational formulation \eqref{eq:def_z_phipsi} for $z_{u,\phi,\psi}$ :
\begin{align*}
\langle F'(y),z_{u,\phi,\psi}\rangle
&=
-\int_\Omega u\nabla p\cdot\nabla z_{u,\phi,\psi}\,dx
=
\int_\Omega \psi\,\nabla z_{u, \phi}\cdot\nabla p\,dx
+
\int_\Omega \phi\,\nabla z_{u, \psi}\cdot\nabla p\,dx.
\end{align*}
Substituting this into  \eqref{j''-proof} leads to \eqref{j''}.

The second-order expansion
\[ \lim_{\substack{\eta\to 0\\ u+\eta\in\feas}}
\frac{
\bigl|j(u+\eta)-j(u)-j'(u)[\eta]-\frac12\,j''(u)[\eta,\eta]\bigr|
}{
\|\eta\|_{L^{\tilde q}(\Omega)}^2
}
=0
\]
is the consequence of twice Fr\'echet differentiability of $j$ on
$(\feas,\|\cdot\|_{L^{\tilde q}(\Omega)})$ for feasible variations.

Finally, boundedness of the bilinear map $j''(u)$ follows from the chain-rule representation,
the boundedness of $F''(y)$ as an operator
$
F''(y)\in \mathcal L\bigl(W^{1,2}_{\tilde\Gamma}(\Omega),W^{-1,2}(\Omega)\bigr),$ 
and the boundedness of $S'(u)$ and $S''(u)$ as multilinear maps with respect to the
$L^{\tilde q}(\Omega)$-norm.
\end{proof}

\section{First-order optimality conditions for
\texorpdfstring{\eqref{eq:p}}{(P)}}
\label{sec:first-order_opt}
For deriving first-order optimality conditions, we consider the following 
abstraction of \eqref{eq:p}:
\begin{gather}\label{eq:r}
\min_{u}\ j(u) + \TV(u)
\ \operatorname{s.t.}\ u \in \feas,
\tag{P$'$}
\end{gather}
where we recover \eqref{eq:p} with the choice $j \coloneqq F \circ S$.
\begin{lemma}
Let $j : L^2(\Omega) \to \R$ be bounded below and lower semicontinuous. Then \eqref{eq:r} admits a solution.
\end{lemma}
\begin{proof}
That follows using the direct method of calculus of variations.
\end{proof}
For \eqref{eq:p} the lower semi-continuity of $j$ follows from the continuity of the solution operator of the state equation asserted
in \cref{lem:S_Lipschitz}. Again, we consider some $s \in (2,\infty)$, where we have in mind the situation
of \cref{rem:index_s} and 
\cref{prp-new:elementary_properties}, and $q \in \big(\frac{2s}{s-2},\infty\big)$.
Then $2 < q < \infty$ and the Hölder conjugate $q'$
satisfies $1 < q' < 2$. Moreover, the convex subdifferential of the indicator function on $\BV(\Omega) \cap L^{q}(\Omega)$
reads:
\[ \partial \delta_{\feas}(\bar{u}) = \{u^* \in (\BV(\Omega) \cap L^{q}(\Omega))^* \,:\, \langle u^*, u - \bar{u} \rangle \le 0
\text{ for all } u \in \feas\}.
\]
In the following, we consider locally optimal solutions to \eqref{eq:r}, where locally means in a neighborhood with respect
to the norm-topology. Since the choice of the topology in $\BV(\Omega)$ (weak$^*$, strict, or norm) usually matters a lot,
we briefly recall that if a function is locally optimal in a weak$^*$-neighborhood or strict neighborhood of some point in
$\BV(\Omega)$, it is also locally optimal in a neighborhood with respect to the norm topology of $\BV(\Omega)$ so that
it makes sense to consider local optimality in this weakest sense.
\begin{lemma}
$\bar{u} \in \feas$ is locally optimal for \eqref{eq:r}
with respect to $\BV(\Omega)$ if and only if it is locally
optimal with respect to $\BV(\Omega) \cap L^{p}(\Omega)$ for $1 < p < \infty$.
\end{lemma}
\begin{proof}
Let $\bar{u} \in \BV(\Omega) \cap \feas$ be locally optimal with respect to
$\BV(\Omega) \cap L^1(\Omega)$, that is, there exists $r > 0$ such that
$j(\bar{u}) + \TV(\bar{u}) \le j(u) + \TV(u)$ 
holds for all $u \in \feas$ with $\max\{\TV(u - \bar{u}), \|u - \bar{u}\|_{L^1}\} \le r$. 
The continuous embedding $L^p(\Omega) \hookrightarrow L^1(\Omega)$ gives some $c > 0$
such that
\[ \max\{\TV(u - \bar{u}),\|u - \bar{u}\|_{L^1}\}
   \le \max\{c,1\} \max\{\TV(u - \bar{u}),\|u - \bar{u}\|_{L^p}\}.
\]
Consequently, $u \in \feas$ with $\max\{\TV(u - \bar{u}),\|u - \bar{u}\|_{L^p}\} \le \frac{r}{\max\{1,c\}}$
implies $j(\bar{u}) + \TV(\bar{u}) \le j(u) + \TV(u)$.
For the reverse direction, we observe that the box constraints imply
\[ \|u - \bar{u}\|_{L^p}^p \le \|u - \bar{u}\|_{L^1}\|(u - \bar{u})^{p - 1}\|_{L^\infty}
   \le \|u - \bar{u}\|_{L^1}(u_u - u_\ell)^{p - 1}.
\]
Consequently, for all $u \in \feas$ with $\max\{\TV(u - \bar{u}),\|u - \bar{u}\|_{L^1}\} \le r$
we have
\begin{align*}
\hspace{1em}&\hspace{-1em} \max\{\TV(u - \bar{u}),\|u - \bar{u}\|_{L^p}\}\\
&\le 
\max\{1,u_u - u_\ell\}^{\frac{p - 1}{p}}\max\{\TV(u - \bar{u}),\|u - \bar{u}\|_{L^1}^\frac{1}{p}\} \\
&\le \max\{1,u_u - u_\ell\}^{\frac{p - 1}{p}}\max\{\|u - \bar{u}\|_{L^1}^\frac{1}{p},\TV(u - \bar{u})^\frac{1}{p},
                                         \|u - \bar{u}\|_{L^1},\TV(u - \bar{u})\}\\
&\le \max\{1,u_u - u_\ell\}^{\frac{p - 1}{p}}\max\{r^{\frac{1}{p}},r\}.
\end{align*}
\end{proof}
We prove our optimality conditions under regularity conditions as are ensured
for our model problem \eqref{eq:p} by \cref{thm:old_regularity_of_first_derivative},
where we recall that $q > 2s/(s - 2)$ implies $q' < 2s/(2 + s)$.
\begin{assumption}\label{ass:j}
Let $j : (\feas, \|\cdot\|_{L^q}) \to \R$ be differentiable on $\feas$
with respect to feasible variations and its derivative be Lipschitz continuous
as a mapping $\nabla_R j : (\feas, \|\cdot\|_{L^q}) \to L^{q'}(\Omega)$.
\end{assumption}
\begin{lemma}
Let \cref{ass:j} hold.
Let $\bar{u} \in \feas$ be locally optimal for \eqref{eq:r}
with respect to $\BV(\Omega) \cap L^{q}(\Omega)$.
Then there are $\sigma \in (\calM(\Omega,\R^d))^*$ with
$\nabla^* \sigma \in \partial \TV(\bar{u}) \subset (\BV(\Omega) \cap L^q(\Omega))^*$,
where the subdifferential is considered with respect
to $\BV(\Omega) \cap L^q(\Omega)$ and
$\nabla^* : (\calM(\Omega,\R^d))^* \to (\BV(\Omega) \cap L^q(\Omega))^*$ is the Banach space adjoint
of $\nabla : \BV(\Omega) \cap L^q(\Omega) \to \calM(\Omega,\R^d)$,
and $u^* \in \partial \delta_{\feas}(\bar{u})$ such that
\[ \int_{\Omega}-\nabla_R j(\bar{u}) v \dd x
= \langle \sigma, \nabla v\rangle_{\calM^*,\calM} 
+ \langle u^*, v \rangle_{(\BV(\Omega) \cap L^{q}(\Omega))^*,\BV(\Omega) \cap L^{q}(\Omega)}
\]
holds for all $v \in \BV(\Omega) \cap L^{q}(\Omega)$.
\end{lemma}
\begin{proof}
The local optimality implies the following first-order optimality condition:
\[ -j'(\bar{u}) \in \partial (\TV + \delta_{\feas})(\bar{u}) 
   \subset (\BV(\Omega) \cap L^{q}(\Omega))^*
\]
We observe by the chain rule for the convex subdifferential (note that $\mu \mapsto \|\mu\|_{\calM}$
is continuous on $\calM (\Omega,\R^d)$) that
$ \eta \in \partial \TV(\bar{u})$, where the subdifferential is considered with respect
to the space $\BV(\Omega) \cap L^q(\Omega)$,  if and only if
$\eta \in \nabla^* (\partial \|\cdot\|_{\calM})(\nabla \bar{u})$
, where the subdifferential is considered with respect to the space $\calM(\Omega,\R^d)$,
so that $\eta = \nabla^*\sigma$ for some $\sigma \in (\calM(\Omega,\R^d))^*$.
The sum rule for subdifferentials gives 
$\partial (\TV + \delta_{\feas})(\bar{u}) 
= \nabla^* (\partial \|\cdot\|_{\calM})(\nabla \bar{u}) + \partial \delta_{\feas}(\bar{u})$
(note that $\tfrac{1}{2}(u_\ell + u_u)$ is in the interior of the effective domain
of $\TV : \BV(\Omega) \cap L^q(\Omega) \to \R$). In combination with the characterization of the convex subdifferential of the indicator
function, we obtain the claim.
\end{proof}
\begin{remark}
We deliberately consider local optimality with respect to the space $\BV(\Omega) \cap L^{q}(\Omega)$, on which $\TV$ is continuous, to apply the sum rule
here. While this setting enables our analysis and is also used in \cite{natemeyer2022penalty}, obtaining a first-order optimality condition for local
minimizers with respect to $L^q(\Omega)$, on which $\TV$ is only lower semi-continuous, would constitute a stronger result. However, this result does not appear to be available with our current approach; 
see also the corresponding discussion in \S1 in \cite{clason2018total}.
\end{remark}
We regularize the problem by means of a positive standard mollifier that induces
the sequence of operators $K_\eta u \coloneqq (k_\eta * u_0)|_\Omega$ for
$\eta \in (0,\infty)$, where $w_0$ is the extension of $w \in L^1(\Omega)$
to $L^1_{\textrm{loc}}(\R)$ by zero outside of $\Omega$ and
$\tilde{w}|_\Omega \in L^1(\Omega)$
the restriction of $\tilde{w} \in L^1_{\textrm{loc}}(\R)$ to $\Omega$.
We define $j_\eta \coloneqq j \circ K_\eta$ and obtain that
$j_\eta : L^1(\Omega) \to \R$ is continuously Fr\'{e}chet differentiable with
$\nabla_R j_\eta(u) = K_\eta^* \nabla_R j(K_\eta u)$, where
we have identified $(L^p(\Omega))^* \cong L^{p'}(\Omega)$ for appropriate choices
of $p \in [1,\infty)$. The mollification regularization of \eqref{eq:r} reads
\begin{gather}\label{eq:r_eta}
\min_{u}\ j_\eta(u) + \TV(u)
\ \operatorname{s.t.}\ u \in \feas.
\tag{P$_\eta'$}
\end{gather}
We recover the corresponding mollification regularization of our model problem 
\eqref{eq:p} as
\begin{gather}\label{eq:p_eta}
\min_{y,u}\ F(y) + \TV(u)
\ \operatorname{s.t.}\
\left\{
\begin{aligned}
- \dvg\big((K_\eta u) \nabla y\big) &= f \text{ in } \Omega, \\
y &= f_D \text{ on } \tilde{\Gamma}, \\
\nabla y \cdot \nu  &= 0  \text{ on } \Gamma, \\
0 < u_\ell \le u(x) &\le u_u \text{ for a.e.\ }x \in \Omega
\end{aligned}
\right.
\tag{P$_\eta$}
\end{gather}
A straightforward $\Gamma$-convergence-type argument (see \cref{lem:ubar_eta_strictly} below for the sketch of
a very similar argument) shows that any sequence of global minimizers
of \eqref{eq:r_eta} as $\eta \searrow 0$ has a convergent subsequence (weakly-$^*$ and strictly
in $\BV(\Omega)$), where the limit is a global minimizer of \eqref{eq:r}.
Due to \cite{natemeyer2022penalty}, it is known that the first-order optimality
condition for \eqref{eq:r_eta} reads
\begin{theorem}
Let \cref{ass:j} hold.
Let $\eta > 0$ and $u_\eta \in \feas$ be locally optimal for \eqref{eq:r_eta} with respect to
$\BV(\Omega) \cap L^{2}(\Omega)$. Then there exist $\sigma_\eta \in L^\infty(\Omega)$ with
$\dvg \sigma_\eta \in L^2(\Omega)$, $\lambda^\ell_\eta \in L^2(\Omega)$, and $\lambda^u_\eta \in L^2(\Omega)$
such that the following optimality system is satisfied:
\begin{align*}
-\dvg \sigma_\eta - \lambda^\ell_\eta + \lambda^u_\eta &= -\nabla_R j_\eta(u_\eta) \\
-\dvg \sigma_\eta &\in \partial \TV (u_\eta) \\
\lambda^\ell_\eta &\ge 0 \\
\lambda^u_\eta &\ge 0 \\
(\lambda^\ell_\eta, u_\ell - u_\eta)_{L^2} &= 0\\
(\lambda^u_\eta, u_\eta - u_u)_{L^2} &= 0.
\end{align*}
\end{theorem}
\begin{proof}
This follows from Theorems 3.12 and 4.3 in \cite{natemeyer2022penalty}.
\end{proof}
For a local minimizer $\bar{u}$ of \eqref{eq:r}, we consider the convex
proximal regularization of the partially linearized problem
\begin{gather}\label{eq:cvx_ubar}
\min_{u \in \feas}\ \int_\Omega \nabla_R j(\bar{u}) (u - \bar{u}) + \frac{1}{q}\|u - \bar{u}\|_{L^{q}}^{q}
+ \TV(u)
\tag{C$_{\bar{u}}$}
\end{gather}
and obtain that $\bar{u}$ is its unique global minimizer:
\begin{lemma}\label{lem:local_minimizers_minimize_cvx_ubar}
Let \cref{ass:j} hold.
If $\bar{u}$ is a local minimizer of \eqref{eq:r} with respect to $\BV(\Omega) \cap L^q(\Omega)$, then $\bar{u}$ is a
global minimizer of \eqref{eq:cvx_ubar} with respect to $\BV(\Omega) \cap L^q(\Omega)$.
\end{lemma}
\begin{proof}
The local optimality of $\bar{u}$ and the differentiability of
$j : (\feas,\|\cdot\|_{L^q}) \to \R$ give that there exists
$r > 0$ such that for all $u \in \feas$ with $\max\{\TV(u - \bar{u}),\|u - \bar{u}\|_{L^{q}}\} \le r$ it holds that
\begin{align*}
\hspace{1em}&\hspace{-1em}
j(\bar{u}) + \int_\Omega \nabla_R j(\bar{u}) (u - \bar{u}) 
+ \frac{1}{q}\|u - \bar{u}\|_{L^{q}}^{q} + \TV(\bar{u})\\
&\ge j(\bar{u}) + \int_\Omega \nabla_R j(\bar{u}) (u - \bar{u}) 
+ o(\|u - \bar{u}\|_{L^{q}}) + \TV(\bar{u})\\
&\ge j(\bar{u}) 
= j(\bar{u}) + \int_\Omega \nabla_R j(\bar{u}) (\bar{u} - \bar{u}) 
+ \frac{1}{q}\|\bar{u} - \bar{u}\|_{L^{q}}^{q} + \TV(\bar{u}).
\end{align*}
Subtracting $j(\bar{u})$ on both sides yields local optimality for \eqref{eq:cvx_ubar}. 
Since the problem is strictly convex, $\bar{u}$ is its unique global minimizer.
\end{proof}
Next, we consider the regularized (convex) problems, where $\nabla_R  j(\bar{u})$
is replaced by $K_\eta^*\nabla j(K_\eta \bar{u})$
\begin{gather}\label{eq:cvx_ubar_eta}
\min_{u \in \feas}\ \int_\Omega \nabla_R j_\eta(\bar{u}) (u - \bar{u}) 
+ \frac{1}{q}\|K_\eta(u - \bar{u})\|_{L^{q}}^{q}
+ \TV(u)
\tag{C$_{\bar{u},\eta}$}
\end{gather}
and obtain the following approximation result.
\begin{lemma}\label{lem:ubar_eta_strictly}
Let \cref{ass:j} hold.
Let $\bar{u}$ solve \eqref{eq:cvx_ubar}. Let $\bar{u}_\eta$ solve \eqref{eq:cvx_ubar_eta} as $\eta \searrow 0$.
Then $\bar{u}_\eta \to \bar{u}$ in $L^{q}(\Omega)$, as well as weakly$^*$ and strictly in $\BV(\Omega)$.
\end{lemma}
\begin{proof}
This follows with a $\Gamma$-convergence-type argument, which we summarize in the interest of completeness.
Because $\bar{u}$ is unique, the claim follows if the sequence $(\bar{u}_\eta)_\eta$ admits accumulation points and all
of them are optimal for \eqref{eq:cvx_ubar}. The existence of accumulation points in $L^1(\Omega)$ follows from the
boundedness in $\BV(\Omega)$ due to the bound constraints and the $\TV$-term in the objective and the feasibility of $\bar{u}$
for all $\eta > 0$. Let $u_\eta \to u^*$ in $L^1(\Omega)$ as $\eta \searrow 0$, we obtain that the first two terms
of the objective of \eqref{eq:cvx_ubar_eta} with $u = u_\eta$ converge to the first terms of \eqref{eq:cvx_ubar} with $u = u^*$.
Moreover, since the feasible sets of \eqref{eq:cvx_ubar_eta} and \eqref{eq:cvx_ubar} coincide and are closed,
the feasibility of a subsequence of $(u_\eta)_\eta$ implies feasibility of the limit $u^*$. Finally,
the $\TV$-seminorm is lower semicontinuous so that $\TV(u^*) \le \liminf_{\eta\searrow 0} \TV(\bar{u}_\eta)$.
Moreover, $u_\eta = \bar{u}$ for all $\eta$ is a feasible recovery sequence with $u_\eta \to \bar{u}$ and
$\TV(u_\eta) \to \TV(\bar{u})$ for the optimal solution $\bar{u}$ to \eqref{eq:cvx_ubar}. Consequently, every accumulation point
of $(\bar{u}_\eta)$ minimizes \eqref{eq:cvx_ubar} and thus coincides with $\bar{u}$ with the famous $\Gamma$-convergence
optimality argument; see, e.g., Proposition 1.18 in \cite{braides2002gamma}.
\end{proof}
As the final auxiliary result bevor we can show our optimality condition,
we now generalize the $L^2$-estimate of Lemma~3.5 in \cite{natemeyer2022penalty}
to an $L^{q'}$-estimate for $q' \in(1,2)$. The main idea is to test the
variational identity \eqref{VI-lemma} below with nonlinear powers of the multipliers, 
specifically with $v = (\lambda^\ell_\eta)^{q' - 1}$ and $v = (\lambda^u_\eta)^{q' - 1}$. 
Since these are not known to be admissible test functions, in particular not in
$H^1(\Omega)$, a priori, we proceed via a truncation argument and monotone convergence
to prove the claim.
\begin{lemma}\label{Lq-mult}
Let \cref{ass:j} hold. Assume $u_\eta\in H^1(\Omega)$ and let  the smooth approximation of the maximum function be defined by
\[
\maxrho(x) \coloneqq
\begin{cases}
\max\{0,x\}, & |x|\ge \tfrac{1}{2\rho},\\[6pt]
\dfrac{\rho}{2}\,\Big(x + \tfrac{1}{2\rho}\Big)^2, & |x| < \tfrac{1}{2\rho},
\end{cases}
\qquad \text{for }\rho>0.
\]
Let
$ \lambda^\ell_\eta(x)\coloneqq\maxrho\big(\rho(u_\ell-u_\eta(x))\big),
\;
\lambda^u_\eta(x)
\coloneqq\maxrho\big(\rho(u_\eta(x)-u_u)\big).$ 
Suppose further that, for all $v\in H^1(\Omega)$,
\begin{equation}\label{VI-lemma}
\int_\Omega \psi'_\varepsilon(\nabla u_\eta)\cdot\nabla v\,d x
-\int_\Omega \lambda^\ell_\eta v\,dx
+\int_\Omega \lambda^u_\eta v\,dx
=\int_\Omega g_\eta v \,d x
\end{equation}
holds with
$
g_\eta \coloneqq \nabla_R j_\eta(\bar u)
+
K_\eta^*\big(|K_\eta(u_\eta-\bar u)|^{q-2}\,K_\eta(u_\eta-\bar u)\big),
$
where $K_\eta$ is a standard mollifier with
$\|K_\eta\|_{\calL(L^r(\Omega))}=\|K_\eta^*\|_{\calL(L^r(\Omega))}=1$
for $1\le r\le\infty$, and $\psi_\varepsilon$ is the $\TV$-smoothing
from \S3 in \cite{natemeyer2022penalty} with $\psi'_\varepsilon(\xi)\cdot\xi\ge 0$
for all $\xi\in\mathbb{R}^n$.
Assume, in addition, that $\rho^2\ge (u_u-u_\ell)^{-1}$ so that $\supp(\lambda^\ell_\eta)\cap\supp(\lambda^u_\eta)=\emptyset$
holds; see Lemma 3.4 in \cite{natemeyer2022penalty}.

Then
\[
\|\lambda^\ell_\eta\|_{L^{q'}(\Omega)}+\|\lambda^u_\eta\|_{L^{q'}(\Omega)}
\le 2\|\nabla j_\eta(\bar u)\|_{L^{q'}(\Omega)} + 2 \|u_\eta-\bar u\|_{L^{q}(\Omega)}^{q - 1}.
\]
\end{lemma}
\begin{proof}
We use a truncation argument and for $k\in\mathbb{N}$ we set
\[
\tilde{T}_k(s) \coloneqq
\begin{cases}
\dfrac{1}{k}, & 0 \le s < \dfrac{1}{k}, \\[4pt]
s, & \dfrac{1}{k} 
\end{cases}
\qquad \text{for }s\in\mathbb{R}.
\]

We approximate $\tilde T_k$ by smooth monotone functions. Let $(\rho_m)_{m\in\mathbb N}$ be a standard family of nonnegative mollifiers and set
$
T_{k,m} \coloneqq \rho_m * \tilde T_k.$
Then $T_{k,m}\in C^\infty(\mathbb R)$, $T_{k,m}$ is nondecreasing,
\[
0\le T_{k,m}' (s)\le 1,
\qquad
\frac1k\le T_{k,m}(s)
\quad\text{for all } s\in\mathbb R,
\]
and $T_{k,m}\to \tilde T_k$ locally uniformly on $\mathbb R$ as $m\to\infty$.

Since $\lambda_\eta^i\ge 0$ and $\lambda_\eta^i\in H^1(\Omega)$ for $i\in\{\ell,u\}$, we define
\[
v_{k,m,\ell} \coloneqq \bigl(T_{k,m}(\lambda^\ell_\eta)\bigr)^{q'-1},
\quad
v_{k,\ell} \coloneqq \bigl(\tilde{T}_{k}(\lambda^\ell_\eta)\bigr)^{q'-1},
\quad
v_{k,m,u} \coloneqq \bigl(T_{k,m}(\lambda^u_\eta)\bigr)^{q'-1},
\quad 
v_{k,u} \coloneqq \bigl(\tilde{T}_{k}(\lambda^u_\eta)\bigr)^{q'-1},
\]
Since $T_{k,m}$ is smooth,  $s \mapsto s^{q' - 1}$ is smooth on $[1/k,\infty)$, and $\lambda_\eta^i\in H^1(\Omega)$, the compositions preserve $H^1$-regularity and thus
$v_{k,m,i}$, $v_{k,i} \in H^1(\Omega)$ for $i\in\{\ell,u\}.$
Moreover,
\[
\nabla v_{k,m,\ell}
=(q'-1)\bigl(T_{k,m}(\lambda^\ell_\eta)\bigr)^{q'-2}
\,T'_{k,m}(\lambda^\ell_\eta)\,\nabla\lambda^\ell_\eta
\quad\text{a.e.\ in }\Omega.
\]
Testing \eqref{VI-lemma} with $v=v_{k,m,\ell}$ yields
\begin{equation}\label{eq:test-vk}
\int_\Omega \psi'_\varepsilon(\nabla u_\eta)\cdot\nabla v_{k,m,\ell}\,dx
-\int_\Omega \lambda^\ell_\eta\,v_{k,m,\ell}\,dx
+\int_\Omega \lambda^u_\eta\,v_{k,m,\ell}\,dx
=
\int_\Omega g_\eta\,v_{k,m,\ell}\,dx.
\end{equation}

It will prove useful below that the definition of $\lambda^\ell_\eta$ gives
\[
  \nabla \lambda^\ell_\eta
  =\maxrho\nolimits'(\rho(u_\ell-u_\eta))\rho(-\nabla u_\eta)
  = -\alpha_\ell(u_\eta) \rho \nabla u_\eta
  \quad \text{with }0\le \alpha_\ell(u_\eta)\le 1,
\]
where $\alpha_\ell(u_\eta) = \maxrho\nolimits'(\rho(u_\ell-u_\eta))$
and the derivative of the smoothed maximum is given by
\[
\maxrho\nolimits'(x)
=
\begin{cases}
1, & |x|\ge \tfrac{1}{2\rho},\\[4pt]
\rho \bigl(x+\tfrac{1}{2\rho}\bigr), & |x|<\tfrac{1}{2\rho}.
\end{cases}
\]
Hence $0 \le \maxrho\nolimits'(x) \le 1$ for all $x\in\mathbb{R}$,
independently of $\rho$. Therefore $\alpha_\ell(u_\eta)$ is uniformly bounded by $1$.
In addition, by assumption on $\psi_\varepsilon$, we have $\psi'_\varepsilon(\xi)\cdot\xi\ge 0$
for all $\xi\in\mathbb{R}^n$. Thus
\[
  \psi'_\varepsilon(\nabla u_\eta)\cdot\nabla v_{k,m,\ell}
  = -(q'-1)\alpha_\ell(u_\eta)\rho\big(T_{k,m}(\lambda^\ell_\eta)\big)^{q'-2}
     T'_{k,m}(\lambda^\ell_\eta)
     \big(\psi'_\varepsilon(\nabla u_\eta)\cdot\nabla u_\eta\big)
  \le 0.
\]
Hence from \eqref{eq:test-vk}, we obtain
\begin{equation}\label{eq:key-ineq-m}
-\int_\Omega \lambda^\ell_\eta\,v_{k,m,\ell}\,dx
+\int_\Omega \lambda^u_\eta\,v_{k,m,\ell}\,dx
\ge
\int_\Omega g_\eta\,v_{k,m,\ell}\,dx
\end{equation}
and passing to the limit $m\to\infty$ yields
\begin{equation}\label{eq:key-ineq}
-\int_\Omega \lambda^\ell_\eta\,v_{k,\ell}\,dx
+\int_\Omega \lambda^u_\eta\,v_{k,\ell}\,dx
\ge
\int_\Omega g_\eta\,v_{k,\ell}\,dx.
\end{equation}

Next, by the support condition
$\supp(\lambda^\ell_\eta)\cap\supp(\lambda^u_\eta)=\emptyset$, we have
that $\lambda^u_\eta>0$ implies $\lambda^\ell_\eta=0$, and therefore
$\tilde T_k(\lambda^\ell_\eta)=\tilde T_k(0)=1/k$ on $\supp(\lambda^u_\eta)$.
Thus
\[
  \int_\Omega \lambda^u_\eta\,v_{k,\ell}\,dx
  = \Big(\frac{1}{k}\Big)^{q'-1}\int_\Omega \lambda^u_\eta\,dx
  \le k^{1-q'}\|\lambda^u_\eta\|_{L^1(\Omega)}
  \to 0
  \qquad \text{as } k\to\infty.
\]

Since $\tilde T_k(\lambda^\ell_\eta)\to \lambda^\ell_\eta$ pointwise as $k\to\infty$ and
the sequence is monotonically increasing, passing to the limit $k\to\infty$ in
\eqref{eq:key-ineq} and applying the monotone convergence theorem yields
\[
  -\int_\Omega (\lambda^\ell_\eta)^{q'}\,dx
  \ge
  \int_\Omega g_\eta\,(\lambda^\ell_\eta)^{q'-1}\,dx .
\]
Therefore,
$\|\lambda^\ell_\eta\|_{L^{q'}(\Omega)}^{q'}
  = \int_\Omega (\lambda^\ell_\eta)^{q'}\,dx
  \le - \int_\Omega g_\eta\,(\lambda^\ell_\eta)^{q'-1}\,dx
  \le \|g_\eta\|_{L^{q'}(\Omega)}
     \|(\lambda^\ell_\eta)^{q' - 1}\|_{L^q(\Omega)}.
$
Since
$  \|(\lambda^\ell_\eta)^{q'-1}\|_{L^q(\Omega)}
  = \|\lambda^\ell_\eta\|_{L^{q'}(\Omega)}^{q'-1},
$
we obtain that
$\|\lambda^\ell_\eta\|_{L^{q'}(\Omega)}^{q'}
  \le
  \|g_\eta\|_{L^{q'}(\Omega)}
  \|\lambda^\ell_\eta\|_{L^{q'}(\Omega)}^{q'-1},
$ and hence
\[
\|\lambda^\ell_\eta\|_{L^{q'}(\Omega)}
\le \|g_\eta\|_{L^{q'}(\Omega)}.
\]

The same argument with $v=v_{k,m,u}$ shows
\[
\|\lambda^u_\eta\|_{L^{q'}(\Omega)}\le \|g_\eta\|_{L^{q'}(\Omega)},
\]
and therefore
\[
 \|\lambda^\ell_\eta\|_{L^{q'}(\Omega)}
 + \|\lambda^u_\eta\|_{L^{q'}(\Omega)}
 \le 2\|g_\eta\|_{L^{q'}(\Omega)} .
\]

Finally,
\[
  g_\eta
  = \nabla j_\eta(\bar u)
  + K_\eta^*\!\big(|K_\eta(u_\eta-\bar u)|^{q-2}K_\eta(u_\eta-\bar u)\big)
\]
and $\|K_\eta\|_{\calL(L^{q}(\Omega))}=\|K_\eta^*\|_{\calL(L^{q'}(\Omega))}=1$ imply
\[
  \|g_\eta\|_{L^{q'}(\Omega)}
  \le \|\nabla j_\eta(\bar u)\|_{L^{q'}(\Omega)}
  + \|(K_\eta(u_\eta-\bar u))^{q-1}\|_{L^{q'}(\Omega)}
  \le \|\nabla j_\eta(\bar u)\|_{L^{q'}(\Omega)}
  + \|u_\eta-\bar u\|_{L^q(\Omega)}^{q-1},
\]
which completes the proof.
\end{proof}

\begin{theorem}\label{thm:stationarity}
Let \cref{ass:j} hold.	
Let $\bar{u}$ be a local minimizer of \eqref{eq:r} with respect to $\BV(\Omega) \cap L^q(\Omega)$.
Then $\bar{u}$ satisfies the following first-order optimality condition.
There exist $\bar{\sigma} \in L^\infty(\Omega,\R^d)$ with $\dvg \bar{\sigma} \in L^{q'}(\Omega)$,
$\bar{\lambda}^\ell \in L^{q'}(\Omega)$, $\bar{\lambda}^u \in L^{q'}(\Omega)$ such that
\begin{align*}
-\dvg \bar{\sigma} - \bar{\lambda}^\ell + \bar{\lambda}^u &= -\nabla_R j(\bar{u}) \\
\bar{\lambda}^\ell &\ge 0 \\
\bar{\lambda}^u &\ge 0 \\
\langle\bar{\lambda}^\ell, u_\ell - \bar{u}\rangle_{L^q,L^{q'}} &= 0\\
\langle\bar{\lambda}^u, \bar{u} - u_u\rangle_{L^q,L^{q'}} &= 0 \\
\|\bar{\sigma}\|_{L^\infty(\Omega)} &= 1 \\
\TV(\bar{u}) &= -\int_{\Omega} \bar{u} \dvg \bar{\sigma} \\
\TV(u) - \TV(\bar{u}) &\ge \int_{\Omega} \dvg \bar{\sigma}(u - \bar{u}) \text{ for all } u \in \feas \text{(!)}.
\end{align*}
\end{theorem}
\begin{proof}
\Cref{lem:local_minimizers_minimize_cvx_ubar} gives that $\bar{u}$
is a global minimizer of \eqref{eq:cvx_ubar}.
Next, we consider the solutions to \eqref{eq:cvx_ubar_eta},
which we know converge to $\bar{u}$ by virtue of \cref{lem:ubar_eta_strictly}.
We note that, because of the mollification regularization, the objective term
\[ (\feas,\|\cdot\|_{L^2(\Omega)}) \ni u \mapsto \int_{\Omega} \nabla_R j_\eta(\bar{u}) (u - \bar{u}) + \frac{1}{q}\|K_\eta (u - \bar{u})\|_{L^{q}(\Omega)}^{q} \in \R\]
is continuously Fr\'{e}chet differentiable for all fixed $\eta > 0$
(and can be extended to a continuously Fr\'{e}chet differentiable function on all of $L^2(\Omega)$). We thus apply Theorems 3.12 and 4.3 in \cite{natemeyer2022penalty}
so that we obtain for all $\eta > 0$ that, in addition to the solution
$u_\eta \in \feas$ of \eqref{eq:cvx_ubar_eta}, there exist
$\sigma_\eta \in L^\infty(\Omega)$ with $\dvg \sigma_\eta \in L^2(\Omega)$, $\lambda^\ell_\eta \in L^2(\Omega)$,
and $\lambda^u_\eta \in L^2(\Omega)$ such that the optimality system
\begin{align*}
-\dvg \sigma_\eta - \lambda^\ell_\eta + \lambda^u_\eta &= -\nabla_R j_\eta(\bar{u}) 
- K_\eta^* \big(K_\eta(u_\eta - \bar{u})|K_\eta(u_\eta - \bar{u})|^{q-2}\big) \\
-\dvg \sigma_\eta &\in \partial \TV (u_\eta) \\
\lambda^\ell_\eta &\ge 0 \\
\lambda^u_\eta &\ge 0 \\
(\lambda^\ell_\eta, u_\ell - u_\eta)_{L^2} &= 0\\
(\lambda^u_\eta, u_\eta - u_u)_{L^2} &= 0
\end{align*}
is satisfied. Inspecting the proofs of Theorems 3.12 and 4.3 in
\cite{natemeyer2022penalty}, we see that there
are paramater sequences $\varepsilon_k \searrow 0$ and $\rho_k \to \infty$ and approximating sequences $\lambda^{\ell,k}_\eta \weakto \lambda^{\ell}_\eta$ in $L^2(\Omega)$,
$\lambda^{u,k}_\eta \weakto \lambda^u_\eta$  in $L^2(\Omega)$,
and $u^k_{\eta} \to u_\eta$ in $L^2(\Omega)$
with $u^k_{\eta} \subset H^1(\Omega)$ so that
the triples $(\lambda^{\ell,k}_\eta,\lambda^{u,k}_\eta,u^k_{\eta})$ satisfy the requirements of 
\cref{Lq-mult} and in particular 
\eqref{VI-lemma}. In addition, the convergence
$u^k_{\eta} \to u_\eta$ also holds in 
$L^q(\Omega)$ because the proof of Lemma 3.9 
in \cite{natemeyer2022penalty} that shows
the convergence in $L^2(\Omega)$ works for
all $L^p(\Omega)$, $p \in [1,\infty)$.

We apply \cref{Lq-mult} and obtain for all $k \in \N$ that the estimate
$\|\lambda^{\ell,k}_\eta\|_{L^{q'}(\Omega)} + \|\lambda^{u,k}_\eta\|_{L^{q'}(\Omega)}
\le 2\|\nabla j_\eta(\bar{u})\|_{L^{q'}(\Omega)} 
   + 2 \|u^k_\eta - \bar{u}\|_{L^{q}(\Omega)}^{q-1}$ holds, which also implies $\lambda^{\ell,k}_\eta \weakto \lambda^\ell_\eta$ and $\lambda^{u,k}_\eta \weakto \lambda^u_\eta$ in $L^{q'}(\Omega)$. Passing to the limit
   	$k \to \infty$ in norm on the right-hand side and using the weak lower semicontinuity of the $L^{q'}(\Omega)$-norm
   	on the left imply
\[
\|\lambda^{\ell}_\eta\|_{L^{q'}(\Omega)} + \|\lambda^{u}_\eta\|_{L^{q'}(\Omega)}
\le 2\|\nabla j_\eta(\bar{u})\|_{L^{q'}(\Omega)} 
+ 2 \|u_\eta - \bar{u}\|_{L^{q}(\Omega)}^{q-1}.\]

 Then we obtain with the same argument as in the proof of
\cref{Lq-mult} and the fact that $\nabla_R j : (\feas, \|\cdot\|_{L^{q}(\Omega)}) \to L^{q'}(\Omega)$ is Lipschitz continuous the inequality
\begin{multline*}
 \|\dvg \sigma_\eta\|_{L^{q'}(\Omega)} + \|\lambda^u_\eta\|_{L^{q'}(\Omega)} + \|\lambda^u_\eta\|_{L^{q'}(\Omega)}\\
\le 3 \big(\|\nabla_R j(\bar{u})\|_{L^{q'}(\Omega)} + 2L\|\bar{u}\|_{L^{q}(\Omega)}  + \lambda(\Omega)^{\frac{q-1}{q}}(u_u - u_\ell)^{q - 1}\big)
\end{multline*}
so that, after passing to a subsequence, we have
\begin{align*}
u_\eta &\to \bar{u} &&\text{in } L^{q}(\Omega),\\
\TV(u_\eta) &\to \TV(\bar{u}) && \text{in }\R\text{; see \cref{lem:ubar_eta_strictly}},\\
\lambda^\ell_\eta &\weakto \bar{\lambda}^\ell &&\text{in } L^{q'}(\Omega),\\
\lambda^u_\eta &\weakto \bar{\lambda}^u &&\text{in } L^{q'}(\Omega).
\end{align*}
Moreover, Proposition 7 from \cite{bredies2016pointwise} gives
$\|\sigma_\eta\|_{L^\infty} = 1$ for all $\eta$ so that we have
\begin{align*}
\sigma_\eta &\weakstarto \bar{\sigma}       &&\text{in } L^\infty(\Omega),\\
\dvg \sigma_\eta &\weakto \dvg \bar{\sigma} &&\text{in } L^{q'}(\Omega)
\end{align*}
because the space $W^{q'}_0(\dvg,\Omega)$ is reflexive. Moreover, we have $\|\bar{\sigma}\|_{L^\infty} = 1$.
We obtain
\begin{align*}
-\dvg \bar{\sigma} - \lambda^\ell + \bar{\lambda}^u &= \nabla_R j(\bar{u}),\\
\bar{\lambda}^\ell &\ge 0,\\
\bar{\lambda}^u &\ge 0,\\
\langle\bar{\lambda}^\ell, u_\ell - \bar{u}\rangle_{L^{q'},L^{q}} &= 0,\\
\langle\bar{\lambda}^u, \bar{u} - u_u\rangle_{L^{q'},L^{q}} &= 0.
\end{align*}
Moreover, we have $u_\eta$, $\bar{u} \in L^\infty(\Omega)$ so that
\[ \TV(\bar{u}) \leftarrow \TV(u_\eta) \underset{\text{Prp.\ 7 in \cite{bredies2016pointwise}}}= -\int_\Omega u_\eta \dvg \sigma_\eta \to -\int_\Omega \bar{u} \dvg \bar{\sigma}.
\]
In particular, $-\dvg \bar{\sigma}$ satisifies the subdifferential characterization
$\TV(\bar{u}) = -\int_\Omega \bar{u}\dvg \bar{\sigma}$, $\|\bar{\sigma}\|_{L^\infty(\Omega)} = 1$
of Proposition 7 in \cite{bredies2016pointwise} except that
$-\dvg \sigma \in L^{q'}(\Omega)$ only (and is not necessarily an element
of the dual space of an $L^p$-space that $\BV(\Omega)$ embeds continuously into).
For all $u \in \feas$, a smooth approximation of $\bar{\sigma}$ \cite{hintermuller2015density}
gives the (restricted) subdifferential inequality
\[ \TV(u) - \TV(\bar{u})
   \ge \int_{\Omega} \dvg \bar{\sigma} u - \TV(\bar{u})
     = \int_{\Omega} \dvg \bar{\sigma}(u - \bar{u}).
\]
\end{proof}

\section{Algorithmic considerations}\label{sec:tr_algorithm}
In this section, we provide a nonsmooth trust-region algorithm with
a convergence analysis for the setting of \cref{sec:first-order_opt}
and \cref{ass:j}, in particular $q \in \big(\frac{2s}{s-2},\infty\big)$.
Regularized coefficient optimization problems like \eqref{eq:r} have been analyzed intensively,
where the case of a regularization term that acts pointwisely is the subject of
\cite{wachsmuth2024control}, the case with $\TV$-regularization is the subject of
\cite{clason2018total,clason2021optimal}, and the case with $H^s$-regularization, $0 < s < 1$, is
the subject of \cite{assmann2013identification}. While \cite{clason2018total,clason2021optimal}
provide semismooth Newton-type or primal-dual splitting proximal splitting-type solution algorithms,
these works do not prove a global convergence results for this setting, which is
in particular nonreflexive in the presence of $\TV$-regularization.

We keep the $\TV$-regularization term in our problem setting and show convergence of
the algorithm to points that satisfy our first-order optimality condition
with respect to $\BV(\Omega) \cap L^q(\Omega)$, equipped with the norm topology,
from \cref{thm:stationarity}. However, our analysis below does not hinge on the specific
presence of the $\TV$-term but also works for general proper, convex, lower semi-continuous
functionals on Banach spaces.

\subsection{Algorithm design and convergence analysis}
As is common, we introduce the criticality measure
\begin{gather}\label{eq:def_calC_and_mu}
\begin{aligned}
\calC(u) \coloneqq 
&&-\min_{v \in \feas}\ &\underbrace{\int_\Omega \nabla_R j(u) (v - u) + \frac{1}{q}\|v - u\|_{L^{q}}^{q}
+ \alpha\TV(v)}_{\eqqcolon m_{u}(v)} - \alpha\TV(u) 
\end{aligned}
\end{gather}
and similar $\calC_\eta$ and $m_{\eta,u}$ for \eqref{eq:cvx_ubar_eta} vs.\ \eqref{eq:cvx_ubar}. The
following lemma motivates to define stationary points as zeros of $\calC$.
\begin{lemma}
\begin{enumerate}
\item $\calC(u) \ge 0$ for all $u \in \feas$.
\item $\calC(\bar{u}) = 0$ if and only if $\bar{u} \in \feas$ solves \eqref{eq:cvx_ubar}.
\item If $\bar{u}$ is a local minimizer of \eqref{eq:r}, then $\calC(\bar{u}) = 0$.
\end{enumerate}
\end{lemma}
\begin{proof}
	This follows from \cref{lem:local_minimizers_minimize_cvx_ubar}
	and the aligned definitions of \eqref{eq:cvx_ubar} and $\calC$.
\end{proof}
\begin{definition}
We say that $u \in \feas$ is \emph{stationary} for \eqref{eq:r} if $u$ solves \eqref{eq:cvx_ubar},
that is, if and only if $\calC(u) = 0$.
\end{definition}
We note that if $u \in \BV(\Omega) \cap \feas$ is not stationary, then the $\Gamma$-convergence-type
result \cref{lem:ubar_eta_strictly} implies $\calC_\eta(u) \to \calC(u) > 0$ as $\eta \searrow 0$. The gist
of trust-region algorithms is to achieve a sufficient objective reduction in the vicinity of iterates $u$
with $\calC(u) > 0$. This is often guaranteed by means of so-called \emph{Cauchy} points, which usually
achieve a predicted reduction of a fraction of the criticality measure multiplied by the trust-region radius;
see, e.g., Lemma 4.3 in \cite{nocedal1999numerical}. In Hilbert space settings, Cauchy points are often steps
along the anti-gradient direction, projected onto the trust region (and possibly further convex constraints).
In convex, nonsmooth Hilbert space settings, an analog to the anti-gradient, that is, a direction of (steepest)
descent generally exists; see, e.g.,  Lemmas 2.75, 2.77 in \cite{ruszczynski2011nonlinear}. It can be
characterized as the projection of zero onto the subdifferential, which renders its explicit computation difficult,
thereby also impairing the application of standard linesearch methods.

It is a common approach to consider (convex) nonsmooths term exactly in trust-region subproblems
\cite{fletcher1982model,yuan1985conditions,hansknecht2024convergence,leyffer2022sequential,manns2023integer}.
This allows the trust-region subproblem to improve over the decrease along the direction of steepest descent to
obtain a locally uniform sufficient objective reduction that leads to global convergence; see (3.7)
in \cite{fletcher1982model} and Chapter 11 in \cite{conn2000trust} for finite-dimensional approaches that
assume a Lipschitz continuous nonsmoothness in a Euclidean space. In our case, we lack such a Lipschitz
property and, in addition, must work in $(\BV(\Omega), \text{weak-}^*)$ and $L^q(\Omega)$ regarding
the asymptotics of our algorithm to have sequential compactness at hand. This is a mismatch
to the topology used for our stationarity condition. To resolve this problem, we will make an algorithmic
compromise in that we reset the trust-region radius in successful iterations. While this may lead to slower
convergence in practice, it allows us to prove convergence to stationary points. We refer to 
\cite{manns2023integer,manns2025homotopy} for a different class of nonsmooth problems,
where the reset technique was applied successfully, and to \cite{manns2026convergence} for a case,
where the reset can be avoided by means of a technical argument in the presence of a $\TV$-term
that acts on integer-valued functions but the argument therein only works for $\Omega \subset \R$.

Our algorithm solves a sequence of subproblems with trust-region models $\tilde{m}_u$ that approximate
$m_u$ and $j$ near given $u$ that satisfy the following assumption.
\begin{assumption}\label{ass:trust-region_model}
For $u \in \feas$, we define $\tilde{m}_u : L^{q}(\Omega) \to \R \cup \{\infty\}$ as
$ \tilde{m}_u(v) \coloneqq q_u(v) + \alpha \TV(v) + \delta_{\feas}(v)$
for all $v \in \feas$, where $q_u : L^{q}(\Omega) \to \R$ is an approximation of $j$
near $u$ that satisfies
\begin{enumerate}
\item $q_u$ is twice differentiable.
\item $q_u(u) = j(u)$.
\item $\nabla_R q_u(u) = \nabla_R j(u)$.
\item\label{itm:regularity_quadratic_model}
 $\sup_{v \in \feas} \|q_u''(v)\|_{\calL(L^{q}, (L^{q})^*)} \le \kappa_{umh} - 1$
for some $\kappa_{umh} \ge 1$ (independent of $u$).
\end{enumerate}
\end{assumption}
\begin{remark}
The note that the assumptions on the model for the smooth part in \cref{ass:trust-region_model} coincide with
the (typical) assumptions AM.1-AM.4 from Chapter 6 in \cite{conn2000trust}. If $q_u$ is actually a
non-linear (typically quadratic) model, \cref{ass:trust-region_model} is non-trivial to satisfy.
E.g., if one wants to use the Hessian for \eqref{eq:p},
we need to strengthen the condition on the integrability
exponent to $q>\frac{4s}{s-2}$ in order to apply
\cref{lem:S_second_derivative,thm:second_order}
with $\tilde q=q$ and deduce the required
boundedness of the second derivative.	
\end{remark}
The trust-region subproblem then reads for some $r \in [1,\infty)$:
\begin{gather}\label{eq:trp}
\begin{aligned}
\min_v\enskip
q_u(v) + \alpha \TV(v)
\enskip
\text{s.t.}\enskip
v \in \feas \text{ and } \|v - u\|_{L^{r}} \le \Delta.
\end{aligned}\tag{TR}
\end{gather}
\begin{lemma}
Let \cref{ass:trust-region_model} hold and $u \in \feas$. Then \eqref{eq:trp} admits a solution.
\end{lemma}
\begin{proof}
This follows similar as the existence of solutions to \eqref{eq:r}.
\end{proof}
We build the trust-region \cref{alg:trm} on the subproblem \eqref{eq:trp}. An iteration $k$ is
\emph{successful} if the sufficient reduction $\sigma_k \ge \underline{\sigma}$ is achieved
so that $u_{k+1}$ becomes the computed trial point $v_k$.
\begin{algorithm}[H]
\caption{Trust-region Algorithm}\label{alg:trm}
\textbf{Input:} $0 < \Delta_{\textrm{res}} = \Delta_0$, $u_0 \in \feas$,
$\underline{\sigma} \in (0,1)$, $\gamma \in (0,1)$.
\begin{algorithmic}[1]
	\For{$k = 0,1,\ldots$}
	\State $v_k \gets $ approximate solution to \eqref{eq:trp} with $\Delta = \Delta_k$ and $u = u_k$.\label{ln:approx_trstep}
	\State $\displaystyle{\sigma_k \gets \frac{j(u_k) + \alpha \TV(u_k) - j(v_k) - \alpha \TV(v_k)}{\tilde{m}_{u_k}(u_k) - \tilde{m}_{u_k}(v_k)}.}$\label{ln:ratio}
	\State $u_{k+1} \gets 
	\begin{cases}
	u_k & \text{ if }\sigma_k < \underline{\sigma}, \\
	v_k & \text{ if }\sigma_k \ge \underline{\sigma}.
	\end{cases}$\label{ln:acceptance}
	\State  $\Delta_{k+1} \gets
	\begin{cases}
	\gamma\Delta_k & \text{ if }\sigma_k < \underline{\sigma}, \\
	\Delta_{\textrm{res}}  & \text{ if }\sigma_k \ge \underline{\sigma}.
	\end{cases}$\label{ln:tr_update}
	\EndFor
\end{algorithmic}
\end{algorithm}
In Line \ref{ln:ratio}, we tacitly set $\sigma_k$ to zero if both numerator and
denominator are zero so that the iteration is unsuccessful. As a consequence,
the sequence of iterates may be eventually constant
if an iterate is stationary.
Regarding the approximation in Line \ref{ln:approx_trstep},
we assume the following.
\begin{assumption}\label{ass:approx_step}
There is $\eta \in (0,1]$ such that all $v_k$ computed in \cref{ln:approx_trstep} satisfy
$v_k \in \BV(\Omega) \cap \feas$ and
\begin{gather}\label{eq:pred_ratio_assumption}
\tilde{m}_{u_k}(u_k) - \tilde{m}_{u_k}(v_k) \ge \eta (\tilde{m}_k(u_k) - \inf \eqref{eq:trp}).
\end{gather}
\end{assumption}
\begin{remark}
We briefly comment how one can ensure \eqref{eq:pred_ratio_assumption} in a numerical realization. If $\tilde{m}_{u_k}$ is convex,
which is typically the case, in particular when BFGS updates are used to assemble the quadratic model, \eqref{eq:trp} can be solved to
approximate global optimality with possible early termination when primal and dual estimates on the optimal objective already
imply \eqref{eq:pred_ratio_assumption}. This can be done using nonlinear programming algorithms, usually ones that are tailored to
quadratic objectives and the trust-region constraint. If one wants to use a non-convex model here, an expensive but general approach
would be to solve the problem by global optimization algorithms like spatial branch-and-bound. We note that in Lipschitz continuous
or differentiable settings, one can use computationally inexpensive so-called Cauchy points to obtain sufficient decrease by performing
a backtracking linesearch; see, e.g., \S6 and \S11 \cite{conn2000trust}. If computationally inexpensive lower bounds on $\inf \eqref{eq:trp}$
are available, one can combine them with Cauchy points to quickly certify \eqref{eq:pred_ratio_assumption} in many iterations and only resort
to global optimization when needed.
\end{remark}

To obtain a sufficient decrease property for the algorithm in our non-Lipschitzian, non-Hilbertian
setting, we observe that if $\calC(u) > 0$ and thus reduction of the objective is possible,
then there must exist a direction that, together with a corresponding element of the subdifferential,
realizes decrease; see Lemma 17.4.3 in \cite{attouch2014variational}. Together with the monotonicity of 
the convex subdifferential, this gives decrease along a line segment, which we will use as a
substitute of the steepest descent direction and Cauchy points. We assert this in \cref{lem:decrease}
below. 
\begin{lemma}\label{lem:decrease}
Let $u_\infty \in \BV(\Omega) \cap \feas$ satisfy $\calC(u_\infty) > 0$. Then there exist
$\varepsilon > 0$ and $\tilde{u} \in \BV(\Omega)
\cap \feas$ such that for all $s \in (0,1)$
it holds that
\[ m_{u_\infty}(u_\infty) - m_{u_\infty}(u_\infty + s(\tilde{u} - u_\infty)) \ge s \varepsilon.
\]
\end{lemma}
\begin{proof}
We first observe that $m_{u_\infty}$; see \eqref{eq:def_calC_and_mu}; is convex since it is the sum of
three convex terms: a linear term, the scaled $q$-th power of the $L^q$-norm, and the scaled $\TV$-seminorm.
We abbreviate $V \coloneqq \BV(\Omega) \cap L^q(\Omega)$.
Since $\calC(u_\infty) > 0$ holds, $u_\infty$ does not minimize the proper, convex, and
lower semicontinuous functional $m_{u_\infty} + \delta_{\feas} : V \to \R \cup \{\infty\}$.
We thus apply  Lemma 17.4.3 in \cite{attouch2014variational} to deduce that there are
$\varepsilon > 0$, $\tilde{u} \in V \cap \feas$, and
\[ 
\tilde{u}^* \in 
   		\{\nabla_R j(u_\infty)\} 
   		+ \Big\{ (\tilde{u} - u_\infty)\frac{|\tilde{u} - u_\infty|^{q - 2}}{\|\tilde{u} - u_\infty\|^{q - 2}_{L^{q}}}\Big\}
  + \partial\big(\alpha \TV + \delta_{\feas}\big)(\tilde{u})
\]
such that $m_{u_\infty}(\tilde{u}) < m_{u_\infty}(u_\infty)$,
$\langle \tilde{u}^*, u_\infty - \tilde{u}\rangle_{V^*,V} \ge \varepsilon$.
	
Let $u^s = \tilde{u} + (1 - s)(u_\infty - \tilde{u}) = u_\infty + s(\tilde{u} - u_\infty)$ for
some $s \in (0,1)$. Because $u_\infty$, $\tilde{u} \in \feas$, it holds $u^s \in \feas$ for
$s \in (0,1)$, implying $0 \in \partial \delta_{\feas}(u^s)$. 
In addition, the $\TV$-term is bounded in a $V$-neighborhood of $u^s$ so that
$\emptyset \neq \partial\TV(u^s)$. In combination with the differentiability of the first two terms in
$m_{u_\infty}$, the sum rule for convex subdifferentials gives $\emptyset \neq \partial (m_{u_\infty} + \delta_{\feas})(u^s)$.
For every $u^{s,*} \in \partial(m_{u_\infty} + \delta_{\feas})(u^s)$, we deduce
\begin{align*}
m_{u_\infty}(u_\infty) - m_{u_\infty}(u^s)
\ge \langle u^{s,*}, u_\infty - u^s\rangle_{V^*,V} &= 
\langle u^{s,*}, s(u_\infty - \tilde{u}) \rangle_{V^*,V}\\
&= s(1 - s)^{-1} \langle u^{s,*}, u^s - \tilde{u} \rangle_{V^*,V}\\
&\ge s(1 - s)^{-1} \langle \tilde{u}^*, u^s - \tilde{u} \rangle_{V^*,V}\\
&= \langle \tilde{u}^*, s(u_\infty - \tilde{u}) \rangle_{V^*,V} > s\varepsilon,
\end{align*}
where the first inequality is the subdifferential inequality and the second is due to the maximal
monotonicity of the subdifferential; see Theorem 17.4.1 in \cite{attouch2014variational}.
\end{proof}
This decrease transfers to the approximate model and 
in turn allows to show that the reduction is
bounded below by a multiple of the trust-region.
\begin{lemma}\label{lem:reduction_bounded_below}
Let \cref{ass:j,ass:trust-region_model,ass:approx_step} hold with $\kappa_{umh} \ge 1$ uniformly on a subset $\calF \subset \feas$. Let $u_k \in \BV(\Omega) \cap \feas$ satisfy $\calC(u_k) > 0$.
Then, there exist $\varepsilon_1 > 0$ and
$\overline{\Delta}$
such that 
\[ \tilde{m}_{u_k}(u_k) - \tilde{m}_{u_k}(v_k)
\ge \varepsilon_1 \Delta. 
\]
\end{lemma}
\begin{proof}
We abbreviate $u \coloneqq u_k$, $\sigma = \sigma_k$. By virtue of \cref{lem:decrease},
there are $\varepsilon$ and $\tilde{u}$ for the choice $u_\infty = u$ therein. We first analyze
the difference of the model decreases from \cref{lem:decrease} to the model decrease in
the denominator of $\sigma_k$ by means of a Taylor expansion of $q_u$:
\begin{align*}
d_{m,\tilde{m}} &\coloneqq |m_u(u) - m_u(u + s(\tilde{u} - u))
- \tilde{m}_u(u) + \tilde{m}_u(u + s(\tilde{u} - u))|
\\
&= \Big|- \langle j'(u), s(\tilde{u} - u)\rangle_{(L^{q})^*,L^{q}} - q^{-1}\|s(\tilde{u} - u)\|_{L^{q}}^{q}\\
&\quad + \alpha \TV(u)
- \alpha \TV(u + s(\tilde{u} - u))\\
&\quad - q_u(u) - \alpha \TV(u)
+ q_u(u + s(\tilde{u} - u))
+ \alpha \TV(u + s(\tilde{u} - u))\Big|\\
&\le 
\big|
s \langle j'(u), u - \tilde{u} \rangle_{(L^q)^*,L^q} - q_u(u) + q_u(u + s(\tilde{u} - u)) \big|
+ q^{-1}\|s(\tilde{u} - u)\|_{L^q}^{q}\\
&\le  0.5 s^2 \big|\langle q_u''(\xi_u(s)) (\tilde{u} - u), \tilde{u} - u\rangle_{(L^{q})^*,L^{q}}\big|
+ q^{-1}s^q\|\tilde{u} - u\|_{L^{q}}^{q}  \\
&\le 0.5s^2(\kappa_{umh} - 1)\|u - \tilde{u}\|^2_{L^q} + q^{-1}s^q\|\tilde{u} - u\|_{L^{q}}^{q}
\end{align*}
where $\xi_u(s)$ lies in the line segment between $u$ and $u + s(\tilde{u} - u)$.
Since $q > 2$, we can choose $\overline{s} \in (0,1]$ small enough such that $d_{m,\tilde{m}} \le \frac{1}{2} s \varepsilon$
holds for all $s \in (0,\overline{s}]$.

Let $\Delta \le \overline{s}\|u - \tilde{u}\|_{L^r} \eqqcolon \overline{\Delta}$,
we obtain $\|u - (u + \Delta\|u - \tilde{u}\|_{L^r}^{-1})(\tilde{u} - u)\|_{L^r} = \Delta$ so
that $u + \Delta\|u - \tilde{u}\|_{L^r}^{-1}(\tilde{u} - u)$ is feasible for
\eqref{eq:trp}. Combining all of our considerations
, we obtain for $v_k$ 
\begin{align*}
\tilde{m}_u(u) - \tilde{m}_u(v)
&\ge \eta(\tilde{m}_u(u) - \inf \eqref{eq:trp} \\
&\ge \eta\big(\tilde{m}_u(u) - \tilde{m}_u(u
+ \Delta\|u - \tilde{u}\|_{L^r}^{-1}(\tilde{u} - u))\big) \\
&\ge \frac{1}{2}\eta\varepsilon\|u - \tilde{u}\|_{L^r}^{-1} \Delta.
\end{align*}
\end{proof}
This implies that a step is eventually accepted
once the trust-region radius is small enough.
\begin{lemma}\label{lem:eventual_step_acceptance}
Let \cref{ass:j,ass:trust-region_model,ass:approx_step} hold with $\kappa_{umh} \ge 1$ uniformly on a subset $\calF \subset \feas$. Let $r > 0.5q$.
Let $j' : L^q(\Omega) \to (L^q(\Omega))^*$ be
continuously differentiable on $\feas$ with
bounded derivative on $\feas$.
Let $u_k \in \BV(\Omega) \cap \feas$ satisfy $\calC(u_k) > 0$.
Then, there exists $\ell \in \N$ such that $\sigma_{k + \ell} \ge \underline{\sigma}$.
\end{lemma}
\begin{proof}
We perform a Taylor expansion of $q_u$ and $j$ at $u$ to obtain with
\begin{gather}\label{eq:pred_vs_ared}
\begin{aligned}
\hspace{1em}&\hspace{-1em}
|\tilde{m}_{u_k}(u_k) - \tilde{m}_{u_k}(v_k) - j(u_k) - \alpha \TV(u_k) + j(v_k) + \alpha \TV(v_k)| \\
&= |q_{u_k}(v_k) - j(v_k)| \le 
0.5\|q_{u_k}''(\xi_q)\|_{\mathcal{L}(L^q,(L^{q})^*)}
\|d_k\|_{L^{q}}^2 
+ 
0.5\|j''(\xi_j)\|_{\mathcal{L}(L^{q},(L^q)^*)}
\|d_k\|_{L^{q}}^2,
\end{aligned}
\end{gather}
with $d_k \coloneqq u_k - v_k$
for some $\xi_q$ and $\xi_j$ in the line segments 
between $u_k$ and $v_k$.
By means of \cref{ass:trust-region_model} and the
assumption on $j''$ we obtain
$ \sup_{u \in \calF} \|\nabla^2 q_u(\xi_q) \|_{\mathcal{L}(L^{q},L^{q'})} + \|\nabla^2 j(\xi_j)\|_{\mathcal{L}(L^{q},L^{q'})} < \infty. $ 
Furthermore, we can deduce $\|d_k\|_{L^{q}} \le c \|d_k\|_{L^r}$ for $r \ge q$. For $q > r > 0.5 q$,
we obtain
\[ \|d_k\|_{L^{q}}^2 \le \|d_k\|_{L^\infty}^{\frac{2(q - r)}{q}} \|d_k\|_{L^r}^{\frac{2r}{q}}
\le (u_u - u_\ell)^{\frac{2(q - r)}{q}} \Delta^{\frac{2r}{q}}
= o(\Delta). \]
This implies that the difference between
numerator and denominator in $\sigma_{k+\ell}$
tends to zero as $\ell \to \infty$ (until
$\sigma_{k + \ell} \ge \underline{\sigma}$
holds).
Using \cref{lem:decrease},
we obtain that the denominator is always positive
and, in combination with the $o(\Delta)$-approximation
above, the numerator is eventually positive 
so that $\sigma_{k + \ell} \to 1$. This implies
that $\sigma_{k + \ell} \ge \underline{\sigma}$
holds eventually.
\end{proof}
Next, we establish auxiliary results that will allow us to prove that convergence to a stationary point 
may imply a decrease in the algorithm that is bounded
below. To this end, we first estimate the decrease
in $m_{u_\infty}$ that we can bound below by 
\cref{lem:decrease}
if $C(u_\infty) > 0$ for some $u_\infty$
against the decrease in the trust-region model
$\tilde{m}_u$
from $u$ to $u_\infty + s(\tilde{u} - u_\infty)$.
\begin{lemma}\label{lem:mq_mtildeu_approx}
\cref{ass:j,ass:trust-region_model,ass:approx_step} hold with $\kappa_{umh} \ge 1$ uniformly on a subset $\calF \subset \feas$.
Let $j' : L^q(\Omega) \to (L^q(\Omega))^*$ be
continuously differentiable on $\feas$ with
bounded derivative on $\feas$.
Let $u_\infty \in \BV(\Omega) \cap \feas$
satisfy $C(u_\infty) > 0$.
Then, for the $\tilde{u} \in \BV(\Omega) \cap \feas$
asserted in \cref{lem:decrease} and for all $s \in (0,1)$, we obtain
\[ d_{m,\tilde{m}}
   \le c\max\{
   s^2,
   s^q,
   s\|j'(u_\infty) - j'(u)\|_{(L^q)^*},
   \|u_\infty - u\|_{L^q}^2,
   |\TV(u_\infty) - \TV(u)|
   \}
\]
for $d_{m,\tilde{m}} \coloneqq 
|m_{u_\infty}(u_\infty) - m_{u_\infty}(u_\infty + s(\tilde{u} - u_\infty))
- \tilde{m}_u(u) + \tilde{m}_u(u_\infty + s(\tilde{u} - u_\infty))|$ and some $c > 0$.
\end{lemma}
\begin{proof}
We deduce
\begin{align*}
d_{m,\tilde{m}} &= |m_{u_\infty}(u_\infty) - m_{u_\infty}(u_\infty + s(\tilde{u} - u_\infty))
- \tilde{m}_u(u) + \tilde{m}_u(u_\infty + s(\tilde{u} - u_\infty))|
\\
&= \Big|- \langle j'(u_\infty), s(\tilde{u} - u_\infty)\rangle_{(L^{q})^*,L^{q}} - q^{-1}\|s(\tilde{u} - u_\infty)\|_{L^{q}}^{q}\\
&\quad + \alpha \TV(u_\infty)
- \alpha \TV(u_\infty + s(\tilde{u} - u_\infty))\\
&\quad - q_u(u) - \alpha \TV(u)
+ q_u(u_\infty + s(\tilde{u} - u_\infty))
+ \alpha \TV(u_\infty + s(\tilde{u} - u_\infty))\Big|\\
&\le 
\big|
s \langle j'(u_\infty), u_\infty - \tilde{u} \rangle_{(L^q)^*,L^q} - q_u(u) + q_u(u_\infty + s(\tilde{u} - u_\infty)) \big|\\
&\quad + q^{-1}s\|\tilde{u} - u_\infty\|_{L^q}^{q}
+ \alpha|\TV(u_\infty) - \TV(u)| \\
&\le 
\big|
- s \langle j'(u_\infty), \tilde{u} - u_\infty  \rangle_{(L^q)^*,L^q} + 
s \langle j'(u), \tilde{u} - u_\infty\rangle_{(L^q)^*,L^q}\big| \\
&\quad + \big|\langle j'(u), u - u_\infty\rangle_{(L^q)^*,L^q}\big| \\
&\quad + 0.5\langle q_u''(\xi)(u_\infty - u + s(\tilde{u} - u_\infty)), u_\infty - u + s(\tilde{u} - u_\infty)\rangle_{{L^q}^*, L^q} \\
&\quad + q^{-1}s\|\tilde{u} - u_\infty\|_{L^q}^{q}  + \alpha|\TV(u_\infty) - \TV(u)| \\
&\le s \|j'(u_\infty) - j'(u)\|_{(L^q)^*}  c_1 + \|u_\infty - u\|_{L^q} c_2 + (\|u_\infty - u\|_{L^q}^2 + s^2)c_3 \\
&\quad + s^q c_4  + |\TV(u_\infty) - \TV(u)|c_5
\end{align*}
for some $c_1$, $\ldots$, $c_5 > 0$.
\end{proof}
Next, we relate $s$ from \cref{lem:decrease}
for an instationary point to the trust-region
radius.
\begin{lemma}\label{lem:s_trradius_relation}
Let $u$, $\tilde{u}$, $u_\infty \in \BV(\Omega)$
be given. For all $s \in (0,1)$, we set
$\Delta(s) \coloneqq s\|\tilde{u} - u_\infty\|_{L^r} + s^2$. If $\|u - u_\infty\|_{L^r} \le s^2$
holds, then 
\[ \Delta(s) \ge \|u - (u_\infty + s(\tilde{u} - u_\infty))\|_{L^r}
\]
holds. In addition, there exists $s_1 \in (0,1)$ such 
that for all $s \in (0,s_1)$, we have 
\[ s \ge \frac{\Delta(s)}{2\|\tilde{u} - u_\infty\|_{L^r}}
\]
\end{lemma}
\begin{proof}
The first claim follows directly from the assumptions
and the triangle inequality. The second claim follows
from the definition of $\Delta(s)$ and $s^2 = o(s)$
as $s \searrow 0$.
\end{proof}
We are now ready to prove that the iterates of \cref{alg:trm} converge to stationary points.
\begin{theorem}\label{thm:asymptotics}
Let 
\cref{ass:j,ass:trust-region_model,ass:approx_step} 
hold with $\kappa_{umh} \ge 1$ uniformly on a subset 
$\calF \subset \feas$. Let $r > 0.5q$.
Let $j : \feas \to \R$ be bounded below.
Let $j' : L^q(\Omega) \to (L^q(\Omega))^*$ be
continuously differentiable on $\feas$ with
bounded derivative on $\feas$.
Then the sequence of iterates produced by
\cref{alg:trm} is feasible with monotonically
decreasing objective values. Every subsequence
of iterates admits a weak-$^*$ cluster point in $\BV(\Omega)$. Every weak-$^*$ cluster point in $\BV(\Omega)$
is feasible, strict,
also a cluster point in $L^p(\Omega)$ for all $p \in [1,\infty)$,
and stationary.
\end{theorem}
\begin{proof}
First, $u_0 \in \feas$ and 
\cref{ass:approx_step} imply inductively
that $u_k$ is feasible for all iterations $k$.
In addition,
$\tilde{m}_{u_k}(u_k) - \inf \eqref{eq:trp} \ge 0$ 
and \cref{ass:approx_step} imply
$\tilde{m}_{u_k}(u_k) - \tilde{m}_{u_k}(v_k) \ge 0$,
that is, the predicted reduction in the denominator
of $\sigma_k$ is nonnegative for all iterations
$k$. In turn, the acceptance criterion in Line
\ref{ln:acceptance} and the computation of $\sigma_k$ 
in Line \ref{ln:ratio} imply that the actual 
reduction $j(u_k) + \alpha \TV(u_k) - j(u_{k+1}) - 
\alpha \TV(u_{k+1})$ is nonnegative for successful 
iterations ($u_{k+1}$ is set to $v_k$)
and zero for unsuccessful iterations ($u_{k+1}$
is set to $u_k$). As a consequence, the sequence of 
objective values decreases monotonically.
Since, $j$ is bounded below, we obtain that $\sup_{k \in \N} \TV(u_k) \le \alpha^{-1}j(u_0) + \TV(u_0) - \inf_{k \in \N} \alpha^{-1}j(u_k) < \infty$.
Due to the box constraints $u_\ell \le u_k \le u_u$
a.e., we obtain $\sup_{k\in \N} \|u_k\|_{L^1} < \lambda(\Omega)(u_u - u_\ell)$. This means that
the $u_k$ are uniformly bounded in $\BV(\Omega)$ and hence every subsequence admits a weak-$^*$ cluster point
in $\BV(\Omega)$.

Since weak-$^*$ convergence in $\BV(\Omega)$
implies convergence in $L^1(\Omega)$ and thus
pointwise a.e.\ convergence of a subsequence,
every such cluster point is feasible.
In addition, the combination of pointwise a.e.\ 
convergence of a subsequence and the box constraints imply
that we can apply Lebesgue's dominated convergence
theorem to deduce (with a subsequence of subsequence
argument) that every such cluster point is also a cluster
point in $L^p(\Omega)$ for all $p \in [1,\infty)$.

Next, we follow the argument from Theorem 4.23 in \cite{leyffer2022sequential}
and Theorem 6.4	(Outcome 3, part 2) in \cite{manns2023integer} to prove that
weak-$^*$ cluster points are strict. By way of contradiction, we assume that
there are $\{ u_{k_\ell} : \ell \in \N \} \cup \{u_\infty\} \subset \BV(\Omega) \cap L^\infty(\Omega)$
with
$u_{k_\ell} \weakstarto u_\infty$ in $\BV(\Omega)$ and $u_{k_\ell} \to u_\infty$ in $L^p(\Omega)$
for all $p \in [1,\infty)$ such that
\[ 0 < \delta \coloneqq \frac{1}{2}(\liminf_{\ell\to\infty} \TV(u_{k_\ell}) - \TV(u_\infty)). \]
We restrict to a subsubsequence, denoted by the same symbol for ease of notation,
such that $\TV(u_{k_\ell}) - \TV(u_\infty) > \delta$ holds for all $\ell \in \N$.

As in the proof of \cref{lem:eventual_step_acceptance}, a Taylor expansion of $q_{u_{k_\ell}}$ and $j$ at $u_{k_\ell}$
yields that for all $v$ that are feasible for \eqref{eq:trp} at $u_{k_\ell}$, we have
\[
\begin{aligned}
\hspace{1em}&\hspace{-1em}
|\tilde{m}_{u_{k_\ell}}(u_{k_\ell}) - \tilde{m}_{u_{k_\ell}}(v) - j(u_{k_\ell}) - \alpha \TV(u_{k_\ell})
+ j(v) + \alpha \TV(v)| \\
&\le 0.5\|q_{u_{k_\ell}}''(\xi_q)\|_{\calL(L^q,(L^q)^*)}\|d_{k_\ell}\|^2_{L^q} + 0.5\|j''(\xi_j)\|_{\calL(L^q,(L^q)^*)}\|d_{k_\ell}\|^2_{L^q}
= o(\Delta_{k_\ell}).
\end{aligned}
\]
with $d_{k_\ell} \coloneqq u_{k_\ell} - v$ and $\xi_q$ and $\xi_j$ in the line segments between $u_{k_\ell}$ and $v$.
In addition, by continuity of $j$ and $q_{u_{k_\ell}}$, we can choose
$0 < \underline{\Delta} \le 2 \eta \alpha\delta \tfrac{1 - \underline{\sigma}}{3 - \underline{\sigma}}$ small enough such that
$\Delta_{k_\ell} \le \underline{\Delta}$ implies
$|q_{u_{k_\ell}}(u_{k_\ell}) - q_{u_{k_\ell}}(v)| \le \frac{1 - \underline{\sigma}}{3 - \underline{\sigma}}\alpha\delta$
and 
$|j(u_{k_\ell}) - j(v)| \le \frac{1 - \underline{\sigma}}{3 - \underline{\sigma}}\alpha\delta$ hold.
If, in addition, $u_{\infty}$ is feasible for \eqref{eq:trp} at $u_{k_\ell}$, this implies
\[ \tilde{m}_{u_{k_\ell}}(u_{k_\ell}) - \tilde{m}_{u_{k_\ell}}(v_{k_\ell}) \ge \eta\big(-\tfrac{1 - \underline{\sigma}}{3 - \underline{\sigma}}\alpha\delta + \alpha\delta\big)
\]
for $v_{k_\ell}$ computed in Line \ref{ln:approx_trstep} satisfying \cref{ass:approx_step}.
We insert our findings in the ratio in Line \ref{ln:ratio} to obtain for all large enough $\ell$ that
\begin{align*}
\sigma_{k_\ell} &= \frac{j(u_{k_\ell}) + \alpha \TV(u_{k_\ell}) - j(v_{k_\ell}) - \alpha \TV(v_{k_\ell})}{\tilde{m}_{u_{k_\ell}}(u_{k_\ell}) - \tilde{m}_{u_{k_\ell}}(v_{k_\ell})} \ge \frac{\tilde{m}_{u_{k_\ell}}(u_{k_\ell}) - \tilde{m}_{u_{k_\ell}}(v_{k_\ell}) -  o(\Delta_{k_\ell})}{\tilde{m}_{u_{k_\ell}}(u_{k_\ell}) - \tilde{m}_{u_{k_\ell}}(v_{k_\ell})}\\
&\ge \frac{\tilde{m}_{u_{k_\ell}}(u_{k_\ell}) - \tilde{m}_{u_{k_\ell}}(v_{k_\ell}) -  2\eta\alpha\delta\tfrac{1-\underline{\sigma}}{3-\underline{\sigma}} }{\tilde{m}_{u_{k_\ell}}(u_{k_\ell}) - \tilde{m}_{u_{k_\ell}}(v_{k_\ell})}
\ge \frac{1 - 3\tfrac{1-\underline{\sigma}}{3 - \underline{\sigma}}}{1 - \tfrac{1-\underline{\sigma}}{3-\underline{\sigma}}}
= \underline{\sigma},
\end{align*}
where the last inequality follows because $p \mapsto (p - c)/p$ is monotone, where we use $c = 2\tfrac{1-\underline{\sigma}}{3 - \underline{\sigma}}$. This means that in this situation, the trial iterate $v_{k_\ell}$ will be accepted to become the
next iterate $u_{k_\ell + 1}$.

We inspect \cref{alg:trm} and observe that all elements $u_{k_\ell}$ of our subsequence occur finitely many
times consecutively in the sequence $(u_k)_k$. Specifically, there is a first iteration when it occurs, namely the iteration
after it gets accepted and where the trust-region radius is equal to $\Delta_{\text{res}}$, and a last iteration when it occurs, 
namely when the next iterate is accepted and in between the trust-region radius is decreased by $\gamma \in (0,1)$; see Line 
\ref{ln:tr_update}. As a consequence, we extend the subsequence to contain all occurrences of the iterates without
changing the assumed convergence properties. Due to $u_{k_\ell} \to u_\infty$ in $L^q(\Omega)$, we can choose $\ell$ large
enough such that if $\Delta_{k_\ell} = \max\{ \Delta_{\text{res}} \gamma^n : \Delta_{\text{res}} \gamma^n \le \underline{\Delta} \text{ for some } n \in \N \}$ holds, we have that $u_\infty$ is always feasible for \eqref{eq:trp} at $u_{k_\ell}$.
With the argument above, we then know that $\sigma_{k_\ell} \ge \underline{\sigma}$ holds after possibly increasing $\ell$ further, 
implying that a iterate is accepted after at most $\min\{ n \in \N : \Delta_{\text{res}} \gamma^n \le \underline{\Delta} \}$ 
reductions of the trust-region radius for the same iterate. In addition, we obtain that the achieved objective reduction satisfies
\[
j(u_{k_\ell}) + \alpha \TV(u_{k_\ell}) - j(v_{k_\ell}) - \alpha \TV(v_{k_\ell})
\ge \sigma(\tilde{m}_{u_{k_\ell}}(u_{k_\ell}) - \tilde{m}_{u_{k_\ell}}(v_{k_\ell})
\ge \eta\alpha\delta (1  -\tfrac{1 - \underline{\sigma}}{3 - \underline{\sigma}}) \eqqcolon r
\]
upon acceptance of $v_{k_\ell}$. Now, again for all $\ell$ large enough, this inequality also
holds if the new iterate is accepted at $\Delta_{k_\ell} > \max\{ \Delta_{\text{res}} \gamma^n : \Delta_{\text{res}} \gamma^n \le \underline{\Delta} \text{ for some } n \in \N \}$ because  $\underline{\sigma}(\tilde{m}_{u_{k_\ell}}(u_{k_\ell}) - \tilde{m}_{u_{k_\ell}}(v_{k_\ell})$ monotonically decreases with reduced trust-region radius as the feasible set shrinks and
cannot fall below the reduction $r$.
Since this situation holds infinitely many times, we obtain that $j(u_{k_\ell}) \to -\infty$, which is a contradiction
to the assumption that $u_\infty$ is a weak-$^*$ cluster point that is not strict. Hence, every weak-$^*$ cluster
point must be strict.

It remains to show that every such cluster point is
stationary. Let $u_\infty \in \BV(\Omega) \cap \feas$
be such a cluster point with strictly approximating 
sequence of $(u_{k_\ell})_\ell \subset \BV(\Omega) \cap \feas$, that is,
\[ u_{k_\ell} \to u_\infty \text{ in } L^q(\Omega)
\quad\text{and}\quad \TV(u_{k_\ell}) \to \TV(u_\infty).
\]
We assume by way of contradicting that
$C(u_\infty) > 0$ and are going to show that
\[ j(u_{k_\ell}) + \TV(u_{k_\ell})
   - j(v_{k_\ell}) + \TV(v_{k_\ell})
   \eqqcolon \ared(k_\ell)
   \ge \underline{\sigma}
   \pred(k_\ell)
   \coloneqq
   \underline{\sigma}(\tilde{m}_{u_k}(u_k) - \tilde{m}_{u_k}(v_k))
   \ge \underline{\sigma}c^*
\]
for some $c^* > 0$. This then implies infinitely
many successful iterations that
each decrease the objective by at least $c^*$,
which contradicts $j + \alpha \TV$ being bounded
from below. Since $C(u_\infty) > 0$, there are
$\tilde{u} \in \BV(\Omega) \cap \feas$ and $\varepsilon > 0$ asserted by \cref{lem:decrease}.

From \eqref{eq:pred_vs_ared}, it follows
that for all $\delta > 0$, there exists $\Delta^\delta > 0$ such that $\Delta_{k_\ell} \le \Delta^\delta$ implies $|\ared(k_\ell) - \pred(k_\ell)| \le \delta \Delta_{k_\ell}$. We choose
\[ \delta
\coloneqq 
\frac{\eta (1 - \underline{\sigma})\varepsilon}{4\|\tilde{u} - u_\infty\|_{L^r}},
\]
where $\eta > 0$ is the one assumed in 
\cref{ass:approx_step}, so that
$\Delta_{k_\ell} \le \Delta^\delta$
implies
\begin{gather}\label{eq:pred_lb_first_estimate} 
\frac{\eta\Delta_{k_\ell}\varepsilon}{4\|\tilde{u} - u_\infty\|_{L^r}}
\ge \frac{|\ared(k_\ell) - \pred(k_\ell)|}{1 - \underline{\sigma}}
\end{gather}
Next, let $s \in (0,s_1)$ for the $s_1 \in (0,1)$ asserted by \cref{lem:s_trradius_relation}. Since the map $s \mapsto \Delta(s)$ from \cref{lem:s_trradius_relation}
is continuous, we can choose $s^* \in (0,s_1)$ such that $\Delta(s^*) \coloneqq \max \big\{ \gamma^\ell \Delta_{\text{res}} : \ell \in \N \text{ and } \gamma^\ell \Delta_{\text{res}} \le \min\{\Delta(s_1),\Delta^\delta\}\big\}$
so that $\Delta(s^*)$ is the trial trust-region radius
of \cref{alg:trm} after $\log_{\gamma}(\Delta(s^*))$
unsuccessful iterations.

We now consider that $\ell_0 \in \N$ large enough such
that for all $\ell \ge \ell_0$, we have
\begin{gather}\label{eq:ukl_uinfty_close}
\max\big\{\|j'(u_{k_\ell}) - j'(u_{\infty})\|_{(L^q)^*},
\|u_{k_\ell} - u_\infty\|_{L^q}^2,
\|u_{k_\ell} - u_\infty\|_{L^r}^{\frac{1}{2}},
|\TV(u_k) - \TV(u_\infty)|^{\frac{1}{2}}\big\}
\le s_*
\end{gather}
for $0 < s_* \coloneqq \min\{s^*, \sqrt{c^{-1} 0.5 \varepsilon s^*}\}$, 
where $c$ is the constant from \cref{lem:mq_mtildeu_approx}.

Consequently, if $\Delta_{k_\ell} \ge \Delta(s^*)$
holds, \cref{lem:s_trradius_relation}
implies that $u_\infty + s^*(\tilde{u} - u_\infty)$
is feasible for \eqref{eq:trp} in iteration $k_\ell$.
In this situation, we have
\begin{align*}
\pred(k_\ell)
&\ge \eta\big(\tilde{m}_{u_{k_\ell}}(u_{k_\ell}) - \inf \eqref{eq:trp}\big) && \text{\scriptsize \cref{ass:approx_step}} \\
&\ge \eta\big(\tilde{m}_{u_{k_\ell}}(u_{k_\ell}) 
- \tilde{m}_{u_k}(u_\infty + s^*(\tilde{u} - u_\infty))\big)
&& \text{\scriptsize feasibility} \\
&\ge \eta\big(m_{u_\infty}(u_\infty) - m_{u_\infty}(u_\infty + s^*(\tilde{u} - u_\infty))
- 0.5 s^* \varepsilon\big) 
&& \text{\scriptsize \eqref{eq:ukl_uinfty_close},
\cref{lem:mq_mtildeu_approx}} \\
&\ge 0.5 \eta s^* \varepsilon
&& \text{\scriptsize \cref{lem:decrease}} \\
&\ge \frac{\eta \varepsilon \Delta(s^*)}{4\|\tilde{u} - u_\infty\|_{L^r}}
&& \text{\scriptsize \cref{lem:s_trradius_relation}} \\
&\ge \frac{|\ared(k_\ell) - \pred(k_\ell)|}{1 - \underline{\sigma}}.
&& \text{\scriptsize \eqref{eq:pred_lb_first_estimate}
if $\Delta_{k_\ell} = \Delta(s^*)$}
\end{align*}
The second to last estimate implies 
$\pred(k_\ell) \ge c^* \coloneqq 0.25 \eta \varepsilon \Delta(s^*) \|\tilde{u} - u_\infty\|_{L^r}^{-1} > 0$ and the last estimate
implies that if $\Delta_{k_\ell} = \Delta(s^*)$,
then
$\ared(k_\ell) \ge \underline{\sigma}\pred(k_\ell)$
and the iteration is successful.

We use these two insights to close the proof with a case distinction. If are infinitely many 
different iterates in the subsequence,
then for all $\ell \ge \ell_0$, there exists
an earliest iteration $p(k_{\ell}) = \min\{k : u_{k_\ell} = u_k\}$ and we have
$\pred(k) \ge c^*$ for all $k \in \{p(k_\ell),\ldots,p(k_\ell) + \log_{\gamma}(\Delta(s^*))\}$.
In addition, $\ared(k) \ge \underline{\sigma}\pred(k)$
for $k \le p(k_\ell) + \log_{\gamma}(\Delta(s^*))$
so that the iteration is successful with a decrease of
at least $c^*$. As a consequence, we obtain infinitely
many successful iterations of this type, implying
$j(u_k) + \alpha \TV(u_k) \to -\infty$, which contradicts $j + \alpha \TV$ being bounded from below.
If there are only finitely many different iterates
in the subsequence, we have $u_{k_\ell} = u_\infty$ 
eventually and
\cref{lem:eventual_step_acceptance} implies
a reduction of the objective and change of the iteration without the possibility to come back
and thus this case
cannot happen.
\end{proof}

\section*{Acknowledgments}
The authors thank the anonymous referees for their helpful comments on the manuscript.

\bibliographystyle{plain}
\bibliography{biblio}
\end{document}